\documentclass[11pt,letterpaper]{amsart}

\usepackage[utf8]{inputenc}
\usepackage{mathbbol}
\usepackage{hyperref}
\usepackage{amsmath}
\usepackage{amssymb}
\usepackage[capitalise]{cleveref}
\usepackage{tikz-cd}
\usepackage{tikz}
\usepackage{aligned-overset}
\usepackage{enumitem}
\usepackage{dynkin-diagrams}
\usepackage{mathrsfs}
\usepackage[margin=3cm]{geometry}

\DeclareMathOperator{\im}{im}

\DeclareMathOperator{\Ind}{Ind}
\DeclareMathOperator{\Span}{Span}
\DeclareMathOperator{\End}{End}
\DeclareMathOperator{\aff}{aff}
\DeclareMathOperator{\Hom}{Hom}
\DeclareMathOperator{\Irr}{Irr}
\DeclareMathOperator{\Perv}{Perv}

\title{An exotic Springer correspondence for $F_4$}
\author{Jonas Antor}
\address{University of Bonn, Endenicher Allee 60, 53115 Bonn, Germany}
\email{antor@math.uni-bonn.de}

\newtheorem{theorem}{Theorem}[section]
\newtheorem{lemma}[theorem]{Lemma}

\newtheorem{remark}[theorem]{Remark}
\newtheorem{definition}[theorem]{Definition}
\newtheorem{corollary}[theorem]{Corollary}

\newtheorem{proposition}[theorem]{Proposition}

\newtheorem{mainthm}{Theorem}
\newtheorem*{theorem*}{Theorem}

\begin{document}
\begin{abstract}
We investigate the structure of the `exotic nilcone' of $F_4$ which is defined by exploiting certain characteristic two phenomena. We show that there are finitely many orbits on this nilcone and construct an associated Springer correspondence. Further to that, we show that all corresponding `exotic Springer fibers' admit an affine paving. We also deduce from this a geometric classification of certain simple modules for the affine Hecke algebra with unequal parameters of type $F_4$.
\end{abstract}
\maketitle

\section{Introduction}
Let $G = F_4(\bar{\mathbb{F}}_2)$ be the simple algebraic group of type $F_4$ in characteristic $2$. The adjoint representation $\mathfrak{g}$ has a $G$-stable subspace $\mathfrak{g}_s \subset \mathfrak{g}$ whose non-zero weights are precisely the short roots \cite{hogeweij1982almost,hiss1984adjungierten}. Consider the $G$-module $V := \mathfrak{g}_s \oplus \mathfrak{g}/\mathfrak{g}_s$. In \cite{antor2025geometric} the author constructed a certain `exotic nilpotent cone' $\mathfrak{N}(V) \subset V$ and an `exotic Springer resolution' $\mu: \tilde{V} \rightarrow \mathfrak{N}(V)$ which can be used to give a geometric realization of the affine Hecke algebra with unequal parameters of type $F_4$. The goal of this paper is to investigate the geometric and combinatorial properties of $\mathfrak{N}(V)$ and to show that there is an associated `exotic Springer correspondence'.

For any $x \in \mathfrak{N}(V)$, let $\mathcal{B}_x := \mu^{-1}(x)$ be the corresponding `exotic Springer fiber' and $A(x) := G_x/G_x^{\circ}$ the corresponding component group. The main results of this paper can be summarized as follows.

\begin{mainthm}\label{thm: A}(\cref{thm: finitely many orbits,thm: exotic Springer correspondence,thm: affine paving})
    \begin{enumerate}
        \item There are only finitely many $G$-orbits in $\mathfrak{N}(V)$. The number of orbits is $24$ and a complete list of orbit representatives together with the structure of their stabilizers is given in \cref{table: nilpotent orbits}.
        \item There is an `exotic Springer correspondence'
        \begin{equation*}
             \Irr(W) \overset{1:1}{\leftrightarrow} \{ (x,\rho) \mid x \in \mathfrak{N}(V) , \rho \in \Irr(A(x)), H_*(\mathcal{B}_x)_{\rho} \neq 0 \}/G.
        \end{equation*}
        Up to conjugation, there is a unique pair $(x,\rho)$ with $H_*(\mathcal{B}_x)_{\rho} = 0$ (namely the pair $(\xi_{17}, sgn)$ in the notation from \cref{table: nilpotent orbits}).
        \item For each $x \in \mathfrak{N}(V)$ the variety $\mathcal{B}_x$ admits an affine paving. In particular, we have $H_{odd}(\mathcal{B}_x) =0$.
    \end{enumerate}
    
\end{mainthm}
A similar exotic Springer correspondence in type $B/C$ was studied by Kato \cite{kato2009exotic,kato2011deformations} and for $G_2$ in \cite{antor2025geometric}. We expect that one should be able to construct similar exotic Springer correspondences for any rooted representation as defined in \cite{antor2025geometric}.

Note that the number of orbits in $\mathfrak{N}(V)$ is different from the number of nilpotent orbits in the classical nilpotent cone $\mathcal{N}\subset \mathfrak{g}$ of $F_4$ (in any characteristic) and thus our exotic Springer correspondence is different from the classical Springer correspondence. Our analysis of $\mathfrak{N}(V)$ also yields several basic results whose analogue for the classical nilpotent cone is well-known. For example, we show that the number of $\mathbb{F}_q$-points in $\mathfrak{N}(V)$ equals $ q^{\dim \mathfrak{N}(V)} = q^{48}$ (\cref{prop: number of F2n points of nilcone}) and that $\mathfrak{N}(V)$ is `$\bar{\mathbb{Q}}_{\ell}$-rationally smooth' (\cref{prop: constant perverse sheaf on nilcone is constant sheaf}). 

\cref{thm: A} also has interesting applications to the representation theory of the affine Hecke algebra of type $F_4$ with unequal parameters $\mathcal{H}^{\aff}_{q_1,q_2}(F_4)$. In fact, it was shown in \cite{antor2025geometric} that one can classify the irreducible representations of this algebra whenever certain geometric conditions hold (these were called (A1)-(A3) in \textit{loc. cit.}). \cref{thm: A} verifies these conditions for representations whose central character is `positive real'. To state the resulting classification more precisely, pick a maximal torus $T \subset G$. To any positive real central character $\chi_a$  we can associate a subtorus $\hat{T}_a \subset T \times \mathbb{G}_m \times \mathbb{G}_m$ (see \cref{section: Springer corr} for more details).
\begin{mainthm}(\cref{thm: exotic DL correspondence})
    For any positive real central character $\chi_a$ there is a bijection
    \begin{equation*}
        \Irr_{\chi_a}(\mathcal{H}^{\aff}_{q_1,q_2}(F_4)) \overset{1:1}{\leftrightarrow} \{ (x,\rho) \mid x \in \mathfrak{N}(V)^{\hat{T}_a}, \rho \in \Irr( A(a,x)), H_*(\mathcal{B}_x^{\hat{T}_a})_{\rho} \neq 0 \}/G^{\hat{T}_a}
    \end{equation*}
    where $A(a,x) := G^{\hat{T}_a}_x/ (G^{\hat{T}_a}_x)^{\circ}$.
\end{mainthm}
This can be thought of as a two-parameter analogue of the Deligne-Langlands correspondence for a single parameter proved in \cite{kazhdan1987proof}. For the choice of parameters $q_1 =q$, $q_2 = q^2$ a similar parameterization can also be obtained from \cite{lusztig1988cuspidal,lusztig1989affine,lusztig1995cuspidal} which was used in the classification of unipotent representations of $p$-adic groups \cite{lusztig1995classification,lusztig2002classification}.

\subsection*{Acknowledgements}
I am thankful to Dan Ciubotaru and Arnaud Eteve for helpful discussions. Some of the computations in this paper were checked with the help of the computer algebra software OSCAR \cite{OSCAR}.
\section{The representation $V$}\label{section: the rep V}
Let $G = F_4(\bar{\mathbb{F}}_2)$ be the exceptional group of type $F_4$ in characteristic $2$. Let $W$ be the Weyl group of $G$. For any $w \in W$ we denote by $\dot{w} \in N_G(T)$ a lift of $w$. Let $\Phi = \Phi_s \sqcup \Phi_l$ be the root system of type $F_4$ where $\Phi_s$ are the short roots and $\Phi_l$ are the long roots. Fix a system of positive roots $\Phi^+ \subset \Phi$ with negative roots $\Phi^- = - \Phi^+$. Let $\alpha_1, \alpha_2, \alpha_3, \alpha_4 \in \Phi^-$ be the (negative) simple roots labelled as follows
\begin{equation*}
    \dynkin[edge length=.75cm,
        labels*={\alpha_1, \alpha_2, \alpha_3, \alpha_4},
    ]F4.
\end{equation*}
For any $a,b,c,d \in \mathbb{Z}$, we write
\begin{equation*}
    \alpha_{abcd} :=  a \alpha_1 + b \alpha_2 +c \alpha_3 + d \alpha_4.
\end{equation*}
The short and long negative roots of $F_4$ are listed in \cref{figure: roots in F4}. Two roots are connected by a non-dashed arrow with label $i$ if they differ by the simple root $\alpha_i$ and by a dashed arrow with label $i$ if they differ by $2 \alpha_i$. The arrow points to the lower root.

\begin{figure}
\caption{Long and short negative roots in $F_4$}
\label{figure: roots in F4}
\begin{equation*}
    \begin{tikzpicture}
    \node at (1,11.5) {Long roots};
    \node at (7,11.5) {Short roots};
    \node[shape=circle,draw=black] (1000) at (0,0) {\tiny 1000};
    \node[shape=circle,draw=black] (0100) at (2,0) {\tiny 0100};
    \node[shape=circle,draw=black] (1100) at (0,1.5) {\tiny 1100};
    \node[shape=circle,draw=black] (0120) at (2,1.5) {\tiny 0120};
    \node[shape=circle,draw=black] (1120) at (0,3) {\tiny 1120};
    \node[shape=circle,draw=black] (0122) at (2,3) {\tiny 0122};
    \node[shape=circle,draw=black] (1220) at (0,4.5) {\tiny 1220};
    \node[shape=circle,draw=black] (1122) at (2,4.5) {\tiny 1122};
    \node[shape=circle,draw=black] (1222) at (1,6) {\tiny 1222};
    \node[shape=circle,draw=black] (1242) at (1,7.5) {\tiny 1242};
    \node[shape=circle,draw=black] (1342) at (1,9) {\tiny 1342};
    \node[shape=circle,draw=black] (2342) at (1,10.5) {\tiny 2342};

    \path [->] (1000) edge node[left] {$2$} (1100);
    \path [->] (0100) edge node[left] {$1$} (1100);
    \path [->,dashed] (0100) edge node[left] {$3$} (0120);
    \path [->,dashed] (1100) edge node[left] {$3$} (1120);
    \path [->] (0120) edge node[left] {$1$} (1120);
    \path [->,dashed] (0120) edge node[left] {$4$} (0122);
    \path [->] (1120) edge node[left] {$2$} (1220);
    \path [->,dashed] (1120) edge node[left] {$4$} (1122);
    \path [->,dashed] (0122) edge node[left] {$1$} (1122);
    \path [->,dashed] (1220) edge node[left] {$4$} (1222);
    \path [->] (1122) edge node[left] {$2$} (1222);
    \path [->,dashed] (1222) edge node[left] {$3$} (1242);
    \path [->] (1242) edge node[left] {$2$} (1342);
    \path [->] (1342) edge node[left] {$1$} (2342);
    
    \node[shape=circle,draw=black] (0010) at (6,0) {\tiny 0001};
    \node[shape=circle,draw=black] (0001) at (8,0) {\tiny 0001};
    \node[shape=circle,draw=black] (0110) at (6,1.5) {\tiny 0110};
    \node[shape=circle,draw=black] (0011) at (8,1.5) {\tiny 0011};
    \node[shape=circle,draw=black] (1110) at (6,3) {\tiny 1110};
    \node[shape=circle,draw=black] (0111) at (8,3) {\tiny 0111};
    \node[shape=circle,draw=black] (1111) at (6,4.5) {\tiny 1111};
    \node[shape=circle,draw=black] (0121) at (8,4.5) {\tiny 0121};
    \node[shape=circle,draw=black] (1121) at (7,6) {\tiny 1121};
    \node[shape=circle,draw=black] (1221) at (7,7.5) {\tiny 1221};
    \node[shape=circle,draw=black] (1231) at (7,9) {\tiny 1231};
    \node[shape=circle,draw=black] (1232) at (7,10.5) {\tiny 1232};

    \path [->] (0010) edge node[left] {$2$} (0110);
    \path [->] (0010) edge node[left] {$4$} (0011);
    \path [->] (0001) edge node[left] {$3$} (0011);
    \path [->] (0110) edge node[left] {$1$} (1110);
    \path [->] (0110) edge node[left] {$4$} (0111);
    \path [->] (0011) edge node[left] {$2$} (0111);
    \path [->] (1110) edge node[left] {$4$} (1111);
    \path [->] (0111) edge node[left] {$1$} (1111);
    \path [->] (0111) edge node[left] {$3$} (0121);
    \path [->] (1111) edge node[left] {$3$} (1121);
    \path [->] (0121) edge node[left] {$1$} (1121);
    \path [->] (1121) edge node[left] {$2$} (1221);
    \path [->] (1221) edge node[left] {$3$} (1231);
    \path [->] (1231) edge node[left] {$4$} (1232);
\end{tikzpicture}
\end{equation*}
\end{figure}
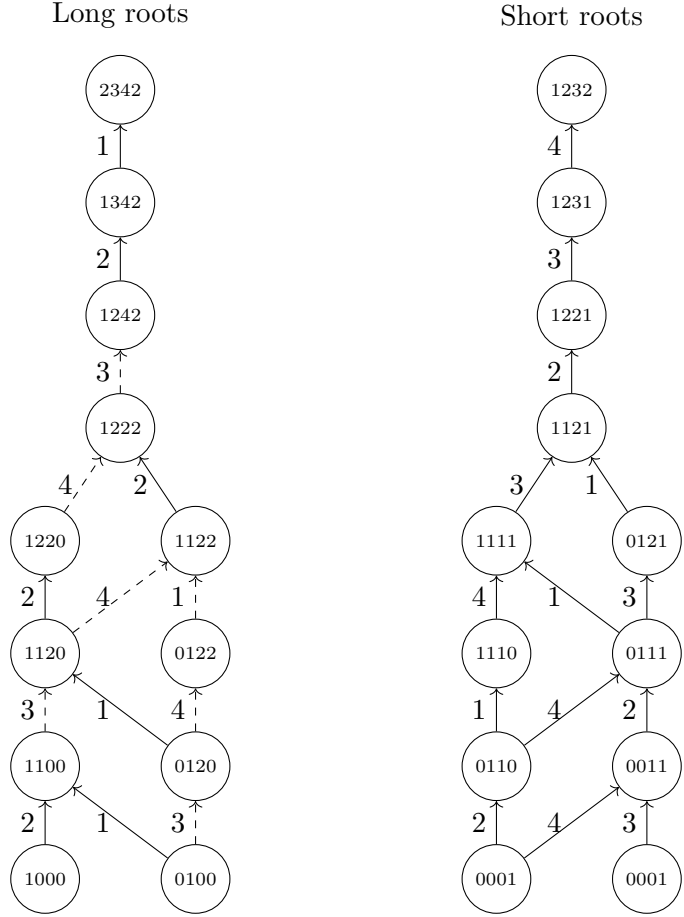

Let $T \subset G$ be a maximal torus and $B \subset G$ the geometric Borel subgroup, i.e. the Borel subgroup containing $T$ whose Lie algebra has non-zero weights $\Phi^-$. The flag variety will be denoted by $\mathcal{B} = G/B$. For any $w \in W$ we write ${}^w B := \dot{w} B \dot{w}^{-1}$. Let $U \subset B$ be the unipotent radical of $B$. We fix a Chevalley basis
\begin{equation*}
    \{X_{\alpha}, H_{\alpha_1}, H_{\alpha_2}, H_{\alpha_3}, H_{\alpha_4} \mid \alpha \in \Phi \}
\end{equation*}
of the adjoint representation $\mathfrak{g}$ (see \cite[Thm. 4.2.1]{carter1989simple} for the definition). The choice of Chevalley basis determines for each $\alpha \in \Phi$ a morphism of algebraic groups (c.f. \cite[Thm. 6.3.1]{carter1989simple})
\begin{equation}\label{eq: root group hom}
    \varphi_{\alpha} : SL_2 \rightarrow G
\end{equation}
which restricts to an isomorphism between $\mathbb{G}_a \cong \left\{ \begin{pmatrix}
    1 & * \\ 0  & 1
\end{pmatrix} \right\}$ and the root group $U_{\alpha} \subset G$. We also denote this isomorphism by
\begin{equation*}
    x_{\alpha}: \mathbb{G}_a \overset{\sim}{\rightarrow} U_{\alpha}.
\end{equation*}
Restricting $\varphi_{\alpha}$ to the subgroup $\mathbb{G}_a \cong \left\{ \begin{pmatrix}
    1 & 0 \\ *  & 1
\end{pmatrix} \right\}$ recovers the isomorphism $x_{-\alpha} : \mathbb{G}_a \overset{\sim}{\rightarrow}  U_{- \alpha}$.

The adjoint representation $\mathfrak{g}$ fits into a short exact sequence
\begin{equation*}
    0 \rightarrow \mathfrak{g}_s \rightarrow \mathfrak{g} \rightarrow \mathfrak{g}/\mathfrak{g}_s \rightarrow 0
\end{equation*}
where $\mathfrak{g}_s$ is the simple $G$-module whose non-zero weights are the short roots and $\mathfrak{g}/\mathfrak{g}_s$ is the simple $G$-module whose non-zero weights are the long roots (c.f. \cite{hogeweij1982almost,hiss1984adjungierten}). The subspace $\mathfrak{g}_s$ is spanned by $H_{\alpha_3}, H_{\alpha_4}$ and the $X_{\alpha}$ with $\alpha \in \Phi_s$. We define
\begin{equation*}
    V := \mathfrak{g}_s \oplus \mathfrak{g}/\mathfrak{g}_s.
\end{equation*}
The inclusion $\mathfrak{g}_s \hookrightarrow \mathfrak{g}$ and projection $\mathfrak{g} \rightarrow \mathfrak{g}/\mathfrak{g}_s$ induce a canonical isomorphism $V_{\alpha} \cong \mathfrak{g}_{\alpha}$ for each $\alpha \in \Phi$. Denote by $v_{\alpha} \in V_{\alpha}$ the element corresponding to $X_{\alpha} \in \mathfrak{g}_{\alpha}$ under this isomorphism. We then have the following explicit formula for the action of root groups on $V$.
\begin{lemma}\label{lemma: action of U on V}
    For any $\alpha, \beta \in \Phi^-$ we have
\begin{equation*}
    x_{\alpha}(t)v_{\beta} = \begin{cases}
        v_{\beta} + t v_{\alpha+\beta} & \alpha + \beta \in \Phi^- \text{ and } \beta, \alpha + \beta \text{ have the same length,} \\
        v_{\beta} + t^2 v_{2\alpha + \beta} & 2\alpha + \beta \in \Phi^- \text{ and } \beta, 2\alpha + \beta \text{ have the same length,}\\
        v_{\beta} & \text{otherwise}.
    \end{cases}
\end{equation*}
\end{lemma}
\begin{proof}
    Let $\beta- p \alpha , ..., \beta + q \alpha$ be the $\alpha$-root string through $\beta$. By the Chevalley formulas (c.f. \cite[§4.3]{carter1989simple}), we have (ignoring signs since we are in characteristic $2$)
    \begin{equation*}
        x_{\alpha}(t) X_{\beta} = \sum_{k = 0}^q t^k {p +k \choose k} X_{k \alpha + \beta}
    \end{equation*}
    which yields  
    \begin{equation}\label{eq: general action of x alpha on v beta}
        x_{\alpha}(t) v_{\beta} = \sum_{\substack{ 0 \le k \le q , \\ \beta \text{ and }  k \alpha + \beta \text{ have} \\ \text{the same length}}} t^k {p +k \choose k} v_{k \alpha + \beta}.
    \end{equation}
    Note that any root string in $F_4$ contains at most three roots (this follows from \cite[Prop. 8.4(e)]{humphreys2012introduction} using the fact that in $F_4$ we always have $\langle \alpha, \check{\beta} \rangle \in \{ -2,-1,0,1,2 \} $). If the root string has length $2$, then both roots in the root string have the same length and if the root string has length three then the two outer roots in the root string are long and the inner root is short (this holds in any root system). Thus, if $p>0$ it can never happen that $k \alpha + \beta$ for $k>0$ is a root of the same length as $\beta$. If $p = 0$ it can only happen that one of $\alpha + \beta$ or $2 \alpha + \beta$ is a root of the same length as $\beta$ but never both. Moreover, in this case ${p + k \choose k} = 1$ so the lemma follows from \eqref{eq: general action of x alpha on v beta}.
\end{proof}
\begin{remark}\label{remark: computation via graph}
    Using \cref{lemma: action of U on V} the action of $x_{\alpha_1}(t),x_{\alpha_2}(t),x_{\alpha_3}(t),x_{\alpha_4}(t)$ on $V^-$ can be read off directly from \cref{figure: roots in F4}: Let $\alpha, \alpha'$ be two nodes in the graph connected by a non-dashed arrow with label $i$ pointing to $\alpha'$. Then $\alpha' = \alpha_i + \alpha$ and $x_{\alpha_i}(t) v_{\alpha} = v_{\alpha} + tv_{\alpha '}$. If the edge is dashed, then $\alpha' = 2\alpha_i + \alpha$ and $x_{\alpha_i}(t) v_{\alpha} = v_{\alpha} + t^2v_{\alpha '}$. In all other cases, $x_{\alpha_i}(t) v_{\alpha} = v_{\alpha}$. Similarly, the action of the simple reflections can be read off from the graph. If $\alpha$ and $\alpha'$ are connected by an arrow (dashed or non-dashed and ignoring orientation) with label $i$, then $s_{\alpha_i}(\alpha) = \alpha'$. Of course, we also have $s_{\alpha_i}(\alpha_i) = - \alpha_i$. In all other cases $s_{\alpha_i}(\alpha) = \alpha$.
\end{remark}
Next, we determine the action of the root groups on the $0$-weight space of $V$. Let 
\begin{align*}
    \mathfrak{h} &:= \Span \{ H_{\alpha_1}, H_{\alpha_2}, H_{\alpha_3}, H_{\alpha_4}\}  \\
\mathfrak{h}_s &:= \mathfrak{h} \cap \mathfrak{g}_s = \Span \{ H_{\alpha_3}, H_{\alpha_4}\}.
\end{align*}
The $0$-weight space of $V$ can then be written as
\begin{equation*}
    V_0 = \mathfrak{h}_s \oplus \mathfrak{h}/\mathfrak{h}_s.
\end{equation*}
By the Chevalley formulas (c.f. \cite[§4.3]{carter1989simple}) we have for any $h \in \mathfrak{h}$
\begin{equation}\label{eq: action of U on h}
    x_{\alpha}(t) h = h + t \alpha(h) X_{\alpha}
\end{equation}
where we view $\alpha$ as an element of $\mathfrak{h}^{\vee} = \Hom(\mathfrak{h}, \bar{\mathbb{F}}_2)$. Note that we can also view any $\alpha \in \Phi_s$ as an element of $\mathfrak{h}_s^{\vee} $ via restriction. Moreover, for each $\alpha \in \Phi_l$ we have $\alpha|_{\mathfrak{h}_s} = 0$ (otherwise we would get $X_{\alpha} \in \mathfrak{g}_s$ by \eqref{eq: action of U on h}). Thus, we can view $\alpha \in \Phi_l$ as an element of $(\mathfrak{h}/\mathfrak{h}_s)^{\vee}$. This defines a $W$-equivariant map
\begin{equation*}
    \Phi \rightarrow \mathfrak{h}_s^{\vee} \oplus (\mathfrak{h}/\mathfrak{h}_s)^{\vee} = V_0^{\vee}.
\end{equation*}
We will denote the image of $\alpha \in \Phi$ under this map again by $\alpha \in V_0^{\vee}$. It then follows from \eqref{eq: action of U on h} that
\begin{equation}\label{eq: action of U alpha on V0}
    x_{\alpha}(t) v = v + t \alpha(v) v_{\alpha}
\end{equation}
for all $v \in V_0$.
\begin{lemma}\label{lemma: alpha non zero on V0}
    The element $\alpha \in V_0^{\vee}$ is non-zero for each $\alpha \in \Phi$.
\end{lemma}
\begin{proof}
    By $W$-equivariance, we may assume that $\alpha$ is simple. Looking at the Cartan matrix of $F_4$, we can find another simple root $\beta$ of the same length as $\alpha$ such that $\langle\alpha, \check{\beta} \rangle = -1$. Hence, we can find an element $v \in V_0$ such that $\alpha(v) = -1 = 1$. Thus, $\alpha \in V_0^{\vee}$ is non-zero.
\end{proof}
\section{The regular semisimple locus}\label{section: rs locus}
The regular semisimple locus in $\mathfrak{g}$ plays an important role in the construction of the classical Springer correspondence (see for example \cite[Chapter 8]{achar2021perverse}). In this section, we define an analogue of the regular semisimple locus for $V$ and use it to prove some basic facts about the exotic nilcone.

Let us first recall the  construction of the exotic Springer resolution for $F_4$ from \cite{antor2025geometric}. Let 
\begin{align*}
    V^- &:= \bigoplus_{\alpha \in \Phi^-} V_{\alpha} \\
    V^{\le 0} &:= V_0 \oplus V^-.
\end{align*}
The exotic Springer resolution is the morphism
\begin{align*}
    \tilde{V} := G \times^B V^- \overset{\mu}&{\rightarrow} V \\
    (g,v) & \mapsto gv
\end{align*}
and the exotic Grothendieck-Springer resolution is the morphism
\begin{align*}
    \tilde{V}^{gs} := G \times^B V^{\le 0} \overset{\mu^{gs}}&{\rightarrow} V\\
    (g,v) & \mapsto gv.
\end{align*}
Define the exotic nilcone to be
\begin{equation*}
    \mathfrak{N}(V) := \im(\mu).
\end{equation*}
Note that $\mu$ is proper, so $\mathfrak{N}(V) \subset V$ is closed.
\begin{lemma}\label{lemma: dimension of nilcone}
    We have $\dim \mathfrak{N}(V) = 48$.
\end{lemma}
\begin{proof}
    The morphism $\tilde{V} \rightarrow \mathfrak{N}(V)$ is surjective and thus $\dim \tilde{V} \ge \dim \mathfrak{N}(V)$. Moreover, a straightforward computation with the Bruhat decomposition shows that $\mu^{-1}(v_{\alpha_1} + v_{\alpha_2} + v_{\alpha_3} + v_{\alpha_4})$ is a singleton (the argument is spelled out in the proof of \cite[Lemma 3.9]{antor2025geometric}). Thus, $\mu$ has a $0$-dimensional fiber and by standard results about fiber dimensions (c.f. \cite[\href{https://stacks.math.columbia.edu/tag/0B2L}{Tag 0B2L}]{stacks-project}) this implies $0 \ge \dim \tilde{V} - \dim \mathfrak{N}(V)$. Thus, we have shown that $\dim \tilde{V} = \dim \mathfrak{N}(V)$. The claim now follows since $ \dim \tilde{V} = \dim G - \dim B + \dim V^- = \dim G - \dim T = 48$.
\end{proof}
We now come to the definition of the regular semisimple locus.
\begin{definition}
    Let $V_0^{rs} := \{v  \in V_0 \mid \alpha(v) \neq 0 \text{ for all } \alpha \in \Phi\}$. We call $V^{rs} := G \cdot V_0^{rs}$ the regular semisimple locus in $V$.
\end{definition}
Next, we establish several basic results about the regular semisimple locus in $V$.
\begin{lemma}\label{lemma: regular semisimple non-empty}
    The set $V^{rs}$ is non-empty.
\end{lemma}
\begin{proof}
    Note that $V_0^{rs} $ is the intersection of the open subsets $\{v \in V_0 \mid \alpha(v) \neq 0 \} \subset V_0$ with $\alpha \in \Phi$ which are non-empty by \cref{lemma: alpha non zero on V0}. By the irreducibility of $V_0$, this implies that $V_0^{rs}$ is non-empty. Thus, $V^{rs}$ is also non-empty.
\end{proof}
\begin{lemma}\label{lemma: stabilizer of rs element}
    For any $v \in V_0^{rs}$ we have $G_v^{\circ} = T$.
\end{lemma}
\begin{proof}
    Let $g \in G_v^{\circ}$. By the Bruhat decomposition we can write $g = \dot{w}ub$ with $b \in B$, $w \in W$ and $u = \prod_{\alpha \in w^{-1}(\Phi^- ) \cap \Phi^+} x_{\alpha}(t_{\alpha})$ for some $t_{\alpha} \in \bar{\mathbb{F}}_2$. If $u \neq 1$, then we can pick $\alpha \in  w^{-1}(\Phi^- ) \cap \Phi^+$ of minimal height such that $t_{\alpha} \neq 0$. Note that we have $\alpha(\dot{w}^{-1} v) = w(\alpha)(v) \neq 0$ since $v \in V_0^{rs}$. Hence, it follows from \eqref{eq: action of U alpha on V0} that the element $u^{-1} \dot{w}^{-1} v $ has a non-zero component in $V_{\alpha}$. This is a contradiction since $u^{-1} \dot{w}^{-1} v = bv \in V^{\le 0}$. Thus, we get $u = 1$. Similarly, if $b \not\in T$, then $bv$ has a non-zero component in some $V_{\alpha}$ with $\alpha \in \Phi^-$. The element $\dot{w}bv$ then has a non-zero component in $V_{w(\alpha)}$, so we can't have $\dot{w}bv = v$. Thus, we have $b \in T$. Hence, $T \subset G_v \subset N_G(T)$ which implies $G_v^{\circ} = T$.
\end{proof}
\begin{lemma}\label{lemma: regular semisimple locus in V le 0}
    We have $V^{\le 0} \cap V^{rs} = V_0^{rs} + V^-$.
\end{lemma}
\begin{proof}
    Let $v = v_0 + v^-$ with $v_0 \in V_0^{rs} $ and $v^- \in V^-$. Repeatedly using \eqref{eq: action of U alpha on V0} we can find an element $u \in U$ with $uv = v_0$. Since $V^{\le 0} \cap V^{rs}$ is $U$-stable, this shows that $V_0^{rs} + V^- \subset V^{\le 0} \cap V^{rs}$. Conversely, let $v \in V^{\le 0} \cap V^{rs}$. Since $v \in V^{rs}$ we have $v = gv_0$ for some $v_0 \in V^{rs}$. By the Bruhat decomposition we can write $g = b \dot{w} b'$ with $b,b' \in B$ and $w \in W$. Then $\dot{w} b'v_0 = b^{-1} v \in V^{\le 0}$. Note that the $V_0$ component of $\dot{w} b'v_0$ is $\dot{w}v_0$, so we get $\dot{w} b'v_0 = \dot{w}v_0 + v'$ for some $v' \in V^-$. Thus, $v = b\dot{w} b'v_0 = b(\dot{w}v_0 + v') =\dot{w}v_0 +v''$ for some $v'' \in V^-$ which shows that $v \in V_0^{rs} + V^-$. Hence, we get $V^{\le 0} \cap V^{rs} = V_0^{rs} + V^-$.
\end{proof}
\begin{lemma}\label{lemma: regular locus is dense}
    $V^{rs} \subset V$ is a dense open subset.
\end{lemma}
\begin{proof}
    Consider the set
    \begin{equation*}
        X:= V^{\le 0 } \cap (V \backslash V^{rs}) = V^{\le 0} \backslash (V^{\le 0} \cap V^{rs}).
    \end{equation*}
    By \cref{lemma: regular semisimple locus in V le 0} the set $V^{\le 0} \cap V^{rs}$ is open in $V^{\le 0}$. Hence, $X$ is closed in $ V^{\le 0}$. This implies that $G \cdot X \subset V$ is closed since it is the image of the proper map $G \times^B X \rightarrow V$. Since $V\backslash V^{rs}$ is $G$-stable, we have
    \begin{equation*}
        (G \cdot X) \cap V^{\le 0} = X.
    \end{equation*}
    By \cite[Lemma 3.9]{antor2025geometric} we have $V = G \cdot V^{\le 0}$ and thus
    \begin{equation*}
         V \backslash (G\cdot X) = G \cdot ( V^{\le 0} \backslash ((G \cdot X) \cap V^{\le 0})) = G \cdot (V^{\le 0} \backslash X) = G \cdot( V^{\le 0} \cap V^{rs} ) = V^{rs}.
    \end{equation*}
    Hence, $V^{rs} \subset V$ is open. It is also non-empty by \cref{lemma: regular semisimple non-empty}. Since $V$ is irreducible, this implies that $V^{rs}$ is dense.
\end{proof}
Next, we show that the Grothendieck-Springer resolution is a $W$-Galois cover when restricted to the regular semisimple locus.
\begin{lemma}\label{lemma: gorth springer resolution over regular semisimple locus}
    There is an isomorphism
    \begin{equation*}
        \tilde{V}^{gs}|_{V^{rs}} \cong G \times^T V^{rs}_0
    \end{equation*}
    over $V^{rs}$. Moreover, the map $\varphi: G \times^T V^{rs}_0 \rightarrow V^{rs}$ is a Galois cover with Galois group $W$ (i.e. $\varphi$ is a finite and étale and $W$ acts simply transitively on the fibers) where the $W$-action on $G \times^T V^{rs}$ is given by $w \cdot(g,v) = (g\dot{w}^{-1}, \dot{w}v)$.
\end{lemma}
\begin{proof}
    It follows from \eqref{eq: action of U alpha on V0} that the canonical map
    \begin{equation*}
        B \times^T V_0^{rs} \cong U \times V_0^{rs} \longrightarrow V_0^{rs} + V^- \overset{\cref{lemma: regular semisimple locus in V le 0}}{=} V^{\le 0} \cap V^{rs}
    \end{equation*}
    is an isomorphism. Hence, we get
    \begin{equation*}
        \tilde{V}^{gs}|_{V^{rs}} \cong G \times^B (V^{\le 0} \cap V^{rs}) \cong G \times^B (B \times^T V_0^{rs}) \cong G \times^T V_0^{rs}
    \end{equation*}
    over $V^{rs}$.

    Let $(g_1, x_1), (g_2, x_2)\in G \times^T V_0^{rs}$ be two elements that map to the same element in $V^{rs}$, i.e. $g_1 x_1 = g_2 x_2$. Then
    \begin{align*}
        T \overset{\cref{lemma: stabilizer of rs element}}{=} G_{x_1}^{\circ} = G_{ g_1^{-1}g_2 x_2}^{\circ} = g_1^{-1}g_2G_{x_2}^{\circ} g_2^{-1}g_1 \overset{\cref{lemma: stabilizer of rs element}}{=} g_1^{-1}g_2Tg_2^{-1} g_1.
    \end{align*}
    Hence, $g_1^{-1}g_2 \in N_G(T)$. Thus, if $w \in W$ is the corresponding Weyl group element, we get $w \cdot (g_2, x_2) = (g_1, x_1)$. Hence, $W$ acts transitively on the fibers of $\varphi$. It also acts freely since $N_G(T)$ acts freely on $G$ by right multiplication. Thus, $W$ acts simply transitively on the fibers.
    
    It remains to show that $\varphi$ is finite and étale. We have seen above that the map $\varphi: G \times^T V^{rs}_0 \rightarrow V^{rs}$ is obtained by restricting the proper morphism $\tilde{V}^{gs} \rightarrow V$ to $V^{rs}$. By standard base change results, this implies that $\varphi$ is proper. We have also seen that $\varphi$ is quasi-finite. These two properties imply that $\varphi$ is finite. Since $G \times^T V_0^{rs}$ and $V^{rs}$ are smooth, this also implies that $\varphi$ is flat (c.f. \cite[Exercise III.9.3.]{hartshorne2013algebraic}). The tangent map of the morphism $G \times V_0^{rs} \rightarrow V^{rs}$ at $(1,v_0)$ is:
    \begin{equation}\label{eq: map from g V0 to V}
        \mathfrak{g} \times V_0 \rightarrow V, (x,v) \mapsto xv_0 + v.
    \end{equation}
    It follows from \eqref{eq: action of U alpha on V0} that $X_{\alpha} v_0 = \alpha(v_0) v_{\alpha}$ and thus the map in \eqref{eq: map from g V0 to V} is surjective. By $G$-equivariance, this implies that $\varphi$ induces a surjection on all tangent spaces. Hence, $\varphi$ is smooth of relative dimension $0$ (see \cite[Prop. 10.4]{hartshorne2013algebraic}), i.e. $\varphi$ is étale.
\end{proof}
\begin{lemma}\label{lemma: springer res is small/semismall}
    The morphism $\tilde{V}^{gs} \rightarrow V$ is small and $\tilde{V} \rightarrow V$ is semismall.
\end{lemma}
\begin{proof}
    The proof is analogous to the classical Springer theory setting (see for example \cite[Lemma 8.2.5]{achar2021perverse}). We sketch the argument that $\tilde{V}^{gs} \rightarrow V$ is small for the convenience of the reader. The fact that $\tilde{V} \rightarrow V$ is semismall can be checked in a similar fashion (see for example the proof of \cite[Lemma 4.10]{antor2025geometric}).
    
    Consider the variety $Z^{gs} := \tilde{V}^{gs} \times_V \tilde{V}^{gs}$. Then $Z^{gs} = \bigsqcup_{w \in W} Z^{gs}_w$ where $Z^{gs}_w$ is the preimage of the $G$-orbit $ \mathcal{O}_w = G \cdot (eB ,\dot{w} B) \in \mathcal{B} \times \mathcal{B}$ under the projection $Z^{gs} \rightarrow \mathcal{B} \times \mathcal{B}$. The fiber of $Z^{gs}_w \rightarrow \mathcal{O}_w$ over $(e, \dot{w})$ is $V^{\le 0 } \cap w(V^{\le 0})$ and there is an isomorphism $ Z^{gs}_w \cong G \times^{B \cap {}^w B} (V^{\le 0 } \cap w ( V^{\le 0}))$ (the canonical map $ G \times^{B \cap {}^w B} (V^{\le 0 } \cap w (V^{\le 0})) \rightarrow Z^{gs}_w $ is bijective by $G$-equivariance and one can check that this is in fact an isomorphism using the open Bruhat cell). In particular, $Z^{gs}_w$ is irreducible and
    \begin{align*}
        \dim Z^{gs}_w &= \dim \mathcal O_w  + \dim V^{\le 0 } \cap w( V^{\le 0}) \\
        & = \dim G - (\dim B - l(w)) + (\dim V^{\le 0} - l(w)) \\
        & = \dim G \\
        &= \dim \tilde{V}^{gs}.
    \end{align*}
    This shows that the $\overline{Z^{gs}_w}$ are the irreducible components of $Z^{gs}$ and they all have dimension $\dim G$. Consider the canonical map
    \begin{equation*}
        \psi : Z^{gs} = \tilde{V}^{gs} \times_V \tilde{V}^{gs} \rightarrow V.
    \end{equation*}
    Each fiber of $Z_w^{gs} \cap \psi^{-1}(V \backslash V^{rs}) \rightarrow \mathcal{O}_w$ is isomorphic to $V^{\le 0} \cap w( V^{\le 0}) \cap (V \backslash V^{rs})$ which has dimension strictly smaller than that of $V^{\le 0} \cap w( V^{\le 0})$ since $V^{rs} \subset V$ is open dense (\cref{lemma: regular locus is dense}). It follows by standard fiber dimension results (see \cite[\href{https://stacks.math.columbia.edu/tag/0B2L}{Tag 0B2L}]{stacks-project}) that
    \begin{equation*}
        \dim Z_w^{gs} \cap \psi^{-1}(V \backslash V^{rs}) <  \dim \mathcal O_w  + \dim V^{\le 0 } \cap w (V^{\le 0}) = \dim Z_w^{gs} 
    \end{equation*}
    and thus
    \begin{equation*}
       \dim \psi^{-1}(V \backslash V^{rs}) < \dim Z^{gs} = \dim \tilde{V}^{gs}. 
    \end{equation*}
    Now let $\{ Y_i \}_{i \in I}$ be a stratification of $V \backslash V^{rs}$ such that the function $x \mapsto \dim \psi^{-1}(x)$ is constant on each $Y_i$. Then by the fiber dimension theorem \cite[\href{https://stacks.math.columbia.edu/tag/0B2L}{Tag 0B2L}]{stacks-project} we have $\dim \psi^{-1}(x) = \dim \psi^{-1}(Y_i) - \dim Y_i $ for some $x \in Y_i$, and thus by definition of $Y_i$ for any $x \in Y_i$. Note that we also have $\psi^{-1}(Y_i) \subset \psi^{-1}(V \backslash V^{rs})$ and thus $\dim \psi^{-1}(Y_i) < \dim \tilde{V}^{gs}$. Moreover, for any $x \in Y_i$ we have $\psi^{-1}(x) \cong (\mu^{gs})^{-1}(x) \times (\mu^{gs})^{-1}(x)$. Combining these results, we get for any $x \in Y_i$
    \begin{equation*}
        2 \dim  (\mu^{gs})^{-1}(x) = \dim \psi^{-1}(x) = \dim \psi^{-1}(Y_i) - \dim Y_i < \dim \tilde{V}^{gs} - \dim Y_i  .
    \end{equation*}
    This proves smallness.
\end{proof}
For any variety $X$ (over $\bar{\mathbb{F}}_2$ or $\mathbb{F}_{2^n}$) we can consider the constructible derived category $D^b_c(X)$ with coefficients in $\mathbb{C} \cong \bar{\mathbb{Q}}_{\ell}$ ($\ell \neq 2$). We write $\textbf{1}_{X} \in D^b_c(X)$ for the constant sheaf and $\omega_X \in D^b_c(x)$ for the dualizing complex. It follows from \cref{lemma: gorth springer resolution over regular semisimple locus} that pushing to constant sheaf $ \textbf{1}_{\tilde{V}^{gs}|_{V^{rs}}} \in D^b_c(\tilde{V}^{gs}|_{V^{rs}})$ along $\mu^{gs}$ yields a local system
\begin{equation*}
    \mathscr{L} := \mu^{gs}_* \textbf{1}_{\tilde{V}^{gs}|_{V^{rs}}} \in Loc(V^{rs}).
\end{equation*}
\begin{corollary}\label{corollary: grothendieck springer sheaf is IC sheaf}
    We have $\mu^{gs}_* \textbf{1}_{\tilde{V}^{gs}}[\dim \tilde{V}^{gs}] = IC(V, \mathscr{L} )$ and $\End( \mu^{gs}_* \textbf{1}_{\tilde{V}^{gs}} ) \cong \bar{\mathbb{Q}}_{\ell}[W]$.
\end{corollary}
\begin{proof}
    This follows from \cref{lemma: springer res is small/semismall} and \cref{lemma: gorth springer resolution over regular semisimple locus} by standard facts about the behavior of perverse sheaves/local systems under small morphisms and Galois covers (see for example \cite[Proposition 3.8.7, Lemma 3.3.3]{achar2021perverse})
\end{proof}
Let $q = 2^n$ for some $n >0$. The $G$-variety $V$ can be defined over $\mathbb{F}_q$ via the Chevalley basis. The varieties $ \tilde{V}, \tilde{V}^{gs}, \mathfrak{N}(V),...$ then have canonical $\mathbb{F}_q$-versions and all morphisms considered so far can be defined over $\mathbb{F}_q$ in a canonical way. We denote by $V_{\circ}, \tilde{V}_{\circ}, \tilde{V}^{gs}_{\circ}, \mathfrak{N}(V)_{\circ}, ...$ the corresponding varieties over $\mathbb{F}_q$. For any variety $X_{\circ}$ defined over $\mathbb{F}_q$ and $\mathcal{F} \in D^b_c(X_{\circ})$ we denote by $\mathcal{F}(n)$ the $n$-the Tate Twist of $\mathcal{F}$. We then have the following $\bar{\mathbb{Q}}_{\ell}$-rational smoothness result.
\begin{proposition}\label{prop: constant perverse sheaf on nilcone is constant sheaf}
    We have $\omega_{\mathfrak{N}(V)_{\circ}}  = \textbf{1}_{\mathfrak{N}(V)_{\circ}}[2 \dim \mathfrak{N}(V)](\dim \mathfrak{N}(V))$.
\end{proposition}
\begin{proof}
    Let us first work over the algebraic closure. The local system $\mathscr{L} = \mu^{gs}_* \textbf{1}_{\tilde{V}^{gs}|_{V^{rs}}} \in Loc(V^{rs})$ is a direct sum of irreducible local system corresponding to the irreducible representations of $W$ and the trivial $W$-representation corresponds to the trivial local system $\textbf{1}_V$. Hence, $\textbf{1}_{V} [\dim V]= IC( V, \textbf{1})$ is a direct summand of $\mu^{gs}_* \textbf{1}_{\tilde{V}^{gs}}[\dim \tilde{V}^{gs}] \overset{\cref{corollary: grothendieck springer sheaf is IC sheaf}}{=} IC(V, \mathscr{L} )$. By \cite[Lemma 3.9]{antor2025geometric} we have $V^{\le 0} \cap \mathfrak{N}(V) = V^-$ which shows that the following diagram is cartesian
    \begin{equation*}
        \begin{tikzcd}
        \tilde{V} \arrow[d, "\mu"] \arrow[r, "\iota'"] & \tilde{V}^{gs} \arrow[d, "\mu^{gs}"] \\
    \mathfrak{N}(V) \arrow[r, "\iota"]                    & V.                          
        \end{tikzcd}
    \end{equation*}
    Hence, we can apply base change to get that a shift of $\textbf{1}_{\mathfrak{N}(V)}= \iota^* \textbf{1}_V$ is a direct summand of
    \begin{equation*}
        \iota^* \mu^{gs}_* \textbf{1}_{\tilde{V}^{gs}} = \mu_* (\iota')^* \textbf{1}_{\tilde{V}^{gs}} = \mu_* \textbf{1}_{\tilde{V}}.
    \end{equation*}
    Note that $\mu_* \textbf{1}_{\tilde{V}}$ is a semisimple complex since $\tilde{V} \overset{\mu}{\rightarrow } \mathfrak{N}(V)$ is semismall by \cref{lemma: springer res is small/semismall}. Moreover, $\textbf{1}_{\mathfrak{N}(V)}$ is indecomposable in $D^b_c(\mathfrak{N}(V))$ (since $End(\textbf{1}_{\mathfrak{N}(V)}) = \mathbb{C}$). Thus, as an indecomposable summand of a semisimple complex, $\textbf{1}_{\mathfrak{N}(V)}$ must be a shift of a simple perverse sheaf (using that $D^b_c(\mathfrak{N}(V))$ is a Krull-Schmidt category). The only possibility for that is that $\textbf{1}_{\mathfrak{N}(V)} [\dim \mathfrak{N}(V)] = IC( \mathfrak{N}(V), \textbf{1})$. 
    
    Now we consider the $\mathbb{F}_q$-structure. Note that being perverse can be checked over the algebraic closure. Hence, $\textbf{1}_{\mathfrak{N}(V)_{\circ}}[\dim \mathfrak{N}(V)] \in D^b_c(\mathfrak{N}(V)_{\circ})$ is also a perverse sheaf. Moreover, the canonical functor $F: \Perv(\mathfrak{N}(V)_{\circ}) \rightarrow \Perv(\mathfrak{N}(V))$ is exact and faithful (the faithfulness follows since $\Perv(\mathfrak{N}(V)_{\circ})$ is equivalent to the category of Weil perverse sheaves; c.f. \cite[Prop. 5.3.9]{achar2021perverse}). Any faithful and exact functor between abelian categories reflects simple objects. Hence, $\textbf{1}_{\mathfrak{N}(V)_{\circ}}[\dim \mathfrak{N}(V)]$ is a simple perverse sheaf. The only possibility for this is that $\textbf{1}_{\mathfrak{N}(V)_{\circ}} [\dim \mathfrak{N}(V)] = IC(\mathfrak{N}(V)_{\circ}, \textbf{1})$. The equality $\omega_{\mathfrak{N}(V)_{\circ}}  = \textbf{1}_{\mathfrak{N}(V)_{\circ}}[2 \dim \mathfrak{N}(V)](\dim \mathfrak{N}(V))$ follows from this by applying Verdier duality (c.f. \cite[Exercise 3.10.3]{achar2021perverse}).
\end{proof}
    Let $F = Fr_q: \mathfrak{N}(V) \rightarrow \mathfrak{N}(V)$ be the (geometric) Frobenius induced by the $\mathbb{F}_q$-structure $\mathfrak{N}(V)_{\circ}$.
\begin{proposition}\label{prop: number of F2n points of nilcone}
    We have $\# \mathfrak{N}(V)^F = q^{\dim \mathfrak{N}(V)} = q^{48}$.
\end{proposition}
\begin{proof}
    Let $X_{\circ}$ be a variety defined over $\mathbb{F}_q$ with structure map $p: X_{\circ} \rightarrow Spec(\mathbb{F}_q)$. Recall that compactly supported cohomology of $X$ (as a Galois-module) is defined as $H^i_c(X) := H^i( p_! \textbf{1}_{X_{\circ}})$ and cohomology is defined as $H^i(X) := H^i(p_* \textbf{1}_{X_{\circ}})$. Thus, if $\omega_{X_{\circ}} = \textbf{1}_{X_{\circ}} [2\dim X] (\dim X)$, we get
    \begin{equation*}
        H_c^i(X)^{\vee} = H^{-i} (p_* \omega_{X_{\circ}}) = H^{-i}(p_*\textbf{1}_{X_{\circ}} [2 \dim X](\dim X)) = H^{2 \dim X - i}(X) (\dim X)
    \end{equation*}
    where the twist $(\dim X)$ on the right hand side means that the Frobenius action is twisted by $q^{- \dim X}$. Applying this for $X = \mathfrak{N}(V)$ (using \cref{prop: constant perverse sheaf on nilcone is constant sheaf}) and passing to the dual, we can identify
    \begin{equation*}
        H_c^i (\mathfrak{N}(V)) \cong H^{2 \dim \mathfrak{N}(V)-i}(\mathfrak{N}(V))^{\vee}(-\dim \mathfrak{N}(V)).
    \end{equation*}
    Since $\mathfrak{N}(V)$ is a $\mathbb{G}_m$-stable closed subvariety of $V$ that contains the origin, the inclusion $ \{ 0 \} \hookrightarrow \mathfrak{N}(V)$ induces an isomorphism $H^*(\mathfrak{N}(V))\cong H^* ( \{ 0\} ) $ by homotopy invariance (see for example \cite[Prop. 1]{springer1984purity}). Thus,
    \begin{equation*}
        H^i_c(\mathfrak{N}(V)) \cong \begin{cases}
            \bar{\mathbb{Q}}_{\ell} (- \dim \mathfrak{N}(V) ) & i = 2 \dim \mathfrak{N}(V), \\ 0 & otherwise.
        \end{cases}
    \end{equation*}
    By the Lefschetz trace formula, this implies
    \begin{equation*}
        \# \mathfrak{N}(V)^F  = \sum_{i \in \mathbb{Z}} (-1)^i Tr(F , H^i_c(\mathfrak{N}(V)) ) = q^{ \dim \mathfrak{N}(V)}\overset{\cref{lemma: dimension of nilcone}}{=} q^{48}.
    \end{equation*}
\end{proof}

\section{The special isogeny}\label{section: special isogeny}
The group $G= F_4(\bar{\mathbb{F}}_2)$ comes with a special isogeny \cite[Exp. 24, p.04]{chevalley1958classification}
\begin{equation*}
    \varphi: G \rightarrow G
\end{equation*}
which corresponds on the character lattice to the map $\varphi^{\#}: X^*(T) \rightarrow X^*(T)$ with
\begin{equation*}
    \varphi^{\#}(\alpha_1) = 2\alpha_4, \quad \varphi^{\#}(\alpha_2) = 2\alpha_3, \quad \varphi^{\#}(\alpha_3) = \alpha_2, \quad \varphi^{\#}(\alpha_4) = \alpha_1.
\end{equation*}
The morphism $\varphi$ is a bijection on $\bar{\mathbb{F}}_2$-points (but not an isomorphism of algebraic groups) and $\varphi \circ \varphi = Fr_2$ is the Frobenius map (with respect to split $\mathbb{F}_2$-structure on $G$).
Note that
\begin{equation*}
    \varphi^{\#} ( \Phi_s) = \Phi_l \quad \text{ and } \quad \varphi^{\#}(\Phi_l) = 2 \Phi_s = Fr_2^{\#}(\Phi_s).
\end{equation*}
Recall that $\mathfrak{g}_s$ and $\mathfrak{g}/\mathfrak{g}_s$ are the simple $G$-modules of lowest weight $\alpha_{1232}$ and $\alpha_{2342}$. We have
\begin{equation*}
    \varphi^{\#}(\alpha_{1232}) = \alpha_{2342}, \quad \varphi^{\#}(\alpha_{2342}) = \alpha_{2462} = Fr_2^{\#} (\alpha_{1232}).
\end{equation*}
Thus, we can pick isomorphisms of $G$-modules
\begin{equation*}
    \mathfrak{g}/\mathfrak{g}_s \overset{\beta_1}{\rightarrow} {}^{\varphi} \mathfrak{g}_s , \quad {}^{Fr_2} \mathfrak{g}_s \overset{\beta_2}{\rightarrow }{}^{\varphi}(\mathfrak{g}/\mathfrak{g}_s)
\end{equation*}
where for any $G$-variety $X$ and a morphism of algebraic groups $\psi : G \rightarrow G$, we denote by ${}^{\psi}X$ the $G$-variety obtained by inflating $X$ along $\psi$. Note that $\mathfrak{g}_s$ is defined over $\mathbb{F}_2$ (via the Chevalley basis), so the corresponding Frobenius map yields a $G$-equivariant bijective morphism of varieties (which is additive but not $\bar{\mathbb{F}}_2$-linear)
\begin{equation*}
    \mathfrak{g}_s \overset{Fr_2}{\longrightarrow} {}^{Fr_2} \mathfrak{g}_s.
\end{equation*}
Thus, we obtain a $G$-equivariant bijective morphism
\begin{align*}
    V = \mathfrak{g}_s \oplus \mathfrak{g}/\mathfrak{g}_s \overset{\begin{pmatrix} 0 & \beta_1 \\ \beta_2 \circ Fr_2 & 0 \end{pmatrix}}{\longrightarrow} {}^{\varphi} \mathfrak{g}_s \oplus {}^{\varphi}(\mathfrak{g}/\mathfrak{g}_s )  = {}^{\varphi}V
\end{align*}
which we denote by
\begin{equation*}
    \psi: V \rightarrow {}^{\varphi} V.
\end{equation*}
We will make use of the map $\psi$ in the following section to determine the structure of some stabilizers via the following lemma.
\begin{lemma}\label{lemma: stabilizer under isogeny iso}
    For any $v \in V$ there is a bijective isogeny of stabilizers
    \begin{equation*}
        G_v \rightarrow G_{\psi(v)}
    \end{equation*}
    where the stabilizers $G_v$ and $G_{\psi(v)}$ are taken with respect to the standard $G$-action on $V$ (i.e. we view $\psi(v)$ as an element of $V$ and not ${}^{\varphi}V$).
\end{lemma}
\begin{proof}
    Let $G_{\psi(v)}'$ be the stabilizer of $\psi(v)$ as an element of ${}^{\varphi} V$. Then the bijective isogeny $\varphi$ restricts to a bijective isogeny $G_{\psi(v)}' \rightarrow  G_{\psi(v)}$. Moreover, since $\psi$ is $G$-equivariant bijection, we have $G_v = G_{\psi(v)}'$.
\end{proof}
We can also compute the action of $\psi$ on weight spaces.
\begin{lemma}\label{lemma: psi on weight spaces}
We have
    \begin{equation*}
    V_{\alpha} =\begin{cases}
        \psi( V_{\varphi^{\#}(\alpha)}) & \alpha \in \Phi_s,\\
        \psi (V_{\tfrac{1}{2} \varphi^{\#}(\alpha) }) & \alpha \in \Phi_l.
        \end{cases}
    \end{equation*}
\end{lemma}
\begin{proof}
    Let $\alpha \in \Phi_s$. Then $\varphi^{\#}(\alpha) \in \Phi_l$. Note that $\beta_1$ is an isomorphism of $G$-modules, so we have
    \begin{equation*}
        \psi(V_{\varphi^{\#}(\alpha)}) = \beta_1(V_{\varphi^{\#}(\alpha)}) = ({}^{\varphi} V)_{\varphi^{\#}(\alpha)} = V_{\alpha}.
    \end{equation*}
    Similarly, if $\alpha \in \Phi_l$, then $\tfrac{1}{2} \varphi^{\#}(\alpha) \in \Phi_s$ and
    \begin{equation*}
        \psi( V_{\tfrac{1}{2} \varphi^{\#}(\alpha)})= \beta_2(Fr_2(V_{\tfrac{1}{2} \varphi^{\#}(\alpha)})) = \beta_2(({}^{Fr_2}V)_{\varphi^{\#}(\alpha)})=  ({}^{\varphi}V)_{\varphi^{\#}(\alpha)} =  V_{\alpha}.
    \end{equation*}
\end{proof}
Note that by \cref{lemma: psi on weight spaces} we have $\psi(V^-) = V^-$. Since $\mathfrak{N}(V) = G \cdot V^-$, we see that $\psi(\mathfrak{N}(V)) = \mathfrak{N}(V)$. Hence, $\psi$ restricts to a $G$-equivariant bijective morphism
\begin{equation*}
    \psi: \mathfrak{N}(V) \rightarrow {}^{\varphi} \mathfrak{N}(V).
\end{equation*}
\section{Exotic nilpotent orbits}
In this section we compute the $G$-orbits on $\mathfrak{N}(V)$. Our strategy is as follows: We first provide a list of orbits in \cref{table: nilpotent orbits} and compute their $\mathbb{F}_q$-points (where $q = 2^n$ for some $n \ge 1$). To show that we have found all the orbits, we compare the number of $\mathbb{F}_q$-points of the orbits with the number of $\mathbb{F}_q$-points in $\mathfrak{N}(V)$ computed in \cref{prop: number of F2n points of nilcone}. A similar strategy has been used to determine nilpotent (co-)adjoint orbits in bad characteristic for some exceptional groups \cite{spaltenstein1983unipotent,holt1985nilpotent, xue2014nilpotent}.
\begin{lemma}\label{lemma: Stabilizers in V and g are the same in some cases}
    Let $\gamma \in \Phi_l$ and $x \in \mathfrak{g}_s$. Then for any $\lambda \in \bar{\mathbb{F}}_2$, the elements $x + \lambda  X_{\gamma} \in \mathfrak{g}$ and $x + \lambda v_{\gamma} \in V$ have the same stabilizer in $G$.
\end{lemma}
\begin{proof}
    The proof is identical to the proof of the analogous result in the $G_2$ case in \cite[Lemma 5.2]{antor2025geometric}. Here is a sketch of the argument: The $\lambda = 0$ case is trivial, so let's assume $\lambda \neq 0$. By $G$-equivariance, we may further assume that $\gamma = \alpha_{2342}$ is the lowest root. A straightforward computation with the Bruhat decomposition (using \cref{lemma: action of U on V}) shows that $G_{X_{\gamma}} = G_{v_{\gamma}} $. Since the $G$-equivariant map $\mathfrak{g} \rightarrow \mathfrak{g}/\mathfrak{g}_s$ sends $x+ \lambda X_{\gamma}$ to $\lambda v_{\gamma}$, we get $G_{x + \lambda X_\gamma} \subset G_{\lambda v_{\gamma}} =G_{ v_{\gamma}} = G_{X_{\gamma}}$ and thus
    \begin{equation*}
        G_{x + \lambda X_\gamma} = G_{x + \lambda X_\gamma} \cap G_{X_{\gamma}} = G_x \cap G_{X_{\gamma}} = G_x \cap G_{v_{\gamma}} = G_{x + \lambda v_{\gamma}}.
    \end{equation*}
\end{proof}
We will make use of the following two elements in $G$
\begin{equation}\label{eq: def of u and s}
    \begin{aligned}
    u &:= x_{\alpha_1}(1) x_{\alpha_4}(1)\\
    s &:= \dot{s}_1 \dot{s}_4
    \end{aligned}
\end{equation}
where we pick
\begin{equation*}
    \dot{s}_i := \varphi_{\alpha_i}\left( \begin{pmatrix}
    0 & 1 \\ 1  & 0
\end{pmatrix} \right)
\end{equation*}
with $\varphi_{\alpha_i} : SL_2 \rightarrow G$ as in \eqref{eq: root group hom}. Note that $u^2 = s^2 = 1$ (using that we are in characteristic $2$). For any $x \in V$ let $A(x) := G_x / G_x^{\circ}$ be the corresponding component group.
\begin{lemma}\label{lemma: stabilizers are as in table}
    The elements $\xi_i$ in \cref{table: nilpotent orbits} have stabilizer dimension, component group and reductive component as listed in the same table. We have $u,s \in G_{\xi_{17}}$ and $A(\xi_{17}) \cong S_3$ is generated by $\bar{u}$ and $\bar{s}$.
\end{lemma}
\begin{proof}
    In principle, all the stabilizers can be computed explicitly by hand similar to the approach in \cite{spaltenstein1984nilpotent}. However, in many cases we can simplify the argument using \cref{lemma: Stabilizers in V and g are the same in some cases} and \cref{lemma: stabilizer under isogeny iso}: For $\xi_1, \xi_2, \xi_3, \xi_4, \xi_5, \xi_8, \xi_{10}, \xi_{18}$ the structure of the stabilizer can be read off directly from \cite[Table 1]{spaltenstein1984nilpotent} using \cref{lemma: Stabilizers in V and g are the same in some cases}. Moreover, it easy to check (e.g. using \cref{remark: computation via graph}) that
    \begin{align*}
        \xi_7 & \sim v_{1231}+v_{1222} \\
        \xi_{11} & \sim v_{0121} + v_{1110} + v_{1222} \\
        \xi_{14}& \sim v_{1110}+v_{0121} + v_{0122}
    \end{align*}
    via an element of $N_G(T)$. The stabilizer structure of these elements can again be read off directly from \cite[Table 1]{spaltenstein1984nilpotent} using \cref{lemma: Stabilizers in V and g are the same in some cases}. Recall from \cref{section: special isogeny} that we have a $G$-equivariant bijective morphism $\psi: \mathfrak{N}(V) \rightarrow {}^{\varphi} \mathfrak{N}(V)$ where $\varphi: G \rightarrow G$ is the special isogeny. Using \cref{lemma: psi on weight spaces}, it is easy to see that
    \begin{equation*}
        \psi(\xi_5) \sim \xi_6, \quad \psi(\xi_8) \sim \xi_9, \quad \psi(\xi_{11}) \sim \xi_{12}, \quad \psi(\xi_{14}) \sim \xi_{15}, \quad \psi(\xi_{18}) \sim \xi_{19}
    \end{equation*}
    via an element in $T$. In each of these cases the structure of the stabilizer for $\psi(\xi)$ can be read off from the structure of the stabilizer of $\xi$ by \cref{lemma: stabilizer under isogeny iso} (note that an isogeny between simple algebraic groups can only exist if the groups have the same type except for type $B_n$ and $C_n$ in characteristic $2$, but $B_n/C_n$ does not occur in the cases at hand).
    
    It remains to compute the stabilizer structure for $\xi_{13}, \xi_{16}, \xi_{17}, \xi_{20}, \xi_{21}, \xi_{22}, \xi_{23}, \xi_{24}$. One can check (either by hand or with the help of a computer) that  except for $\xi_i = \xi_{17}$, each of these $\xi_i$ satisfy $G_{\xi_i} \cap B\dot{w} B = \emptyset$ for $w \neq e$ and thus $G_{\xi_i} = B_{\xi_i}$. The group $B_{\xi_i}$ can be easily computed using \cref{lemma: action of U on V} (by hand or with the help of a computer).
    
    Thus, it only remains to determine the structure of $ G_{\xi_{17}}$. In this case, one can check that $G_{\xi_{17}} \cap B \dot{w} B = \emptyset$ unless $w \in \{e, s \}$. $B_{\xi_{17}}$ can be again computed using \cref{lemma: action of U on V} and one finds that
    \begin{align*}
        \dim B_{\xi_{17}} &= 12\\
        B_{\xi_{17}}/B_{\xi_{17}}^{\circ} &= \langle \bar{u} \rangle \cong \mathbb{Z}/2.
    \end{align*}
    Moreover, a direct computation with \cref{lemma: action of U on V} shows that
    \begin{equation*}
        G_{\xi_{17}} \cap B s B = \{ s, us \} \cdot B_{\xi_{17}}.
    \end{equation*}
    Thus, $G_{\xi_{17}}^{\circ} = B_{\xi_{17}}^{\circ}$ which shows that $\dim G_{\xi_{17}} = 12$. Moreover,
    \begin{equation*}
        A(\xi_{17})= G_{\xi_{17}}/G_{\xi_{17}}^{\circ} = B_{\xi_{17}}/B_{\xi_{17}}^{\circ} \sqcup  (B s B \cap G_{\xi_{17}})/B_{\xi_{17}}^{\circ} = \{e, \bar{u}, \bar{s}, \bar{u} \bar{s} , \bar{s}\bar{u}, \bar{u} \bar{s} \bar{u} \}
    \end{equation*}
    is a group of order $6$ with at least two elements of order two (namely $\bar{u}$ and $\bar{s})$. Thus, $A(\xi_{17})$ has to be isomorphic to $S_3$. Finally, looking at weight spaces, one can see that $T_{\xi_{17}} = \{e\}$ which shows that the reductive part of $G_{\xi_{17}}^{\circ} = B_{\xi_{17}}^{\circ}$ is trivial.
\end{proof}
\begin{table}[ht]
\caption{Orbit representatives in $\mathfrak{N}(V)$ and stabilizer structure}
\label{table: nilpotent orbits}
\centering
\begin{tabular}{l|c|c|c}
Representative $\xi$ & $\dim G_{\xi}$ & $ A(\xi)$  & type of $G_{\xi}^{\circ}/R_u(G_{\xi}^{\circ})$\\ \hline
$\xi_1 = 0$ & $52$ & $1$ & $F_4$ \\
$\xi_2 = v_{2342}$ & $36$ &  $1$ & $C_3$ \\
$\xi_3 = v_{1232}$ & $36$ &  $1$ & $B_3$ \\
$\xi_4 = v_{1232} + v_{2342}$ & $30$ &  $1$ & $B_2$ \\
$\xi_5 = v_{0121} + v_{1111}$ & $28$ &  $1$ & $G_2$ \\
$\xi_6 = v_{1220} + v_{1122}$ & $28$ &  $1$ & $G_2$ \\
$\xi_7 = v_{1221} + v_{1242}$ & $24$ &  $1$ & $A_1 \times A_1$ \\
$\xi_8 = v_{0121} + v_{1111}+v_{2342}$ & $22$ &  $1$ & $A_1$ \\
$\xi_9 = v_{1232} + v_{1220}+v_{1122}$ & $22$ &  $1$ & $A_1$ \\
$\xi_{10} = v_{1110} + v_{0122}$ & $20$ &  $1$ & $B_2$ \\
$\xi_{11} = v_{0121} + v_{1111} + v_{1220}$ & $18$ &  $1$ & $A_1$ \\
$\xi_{12} = v_{0121} + v_{1220} + v_{1122}$ & $18$ &  $1$ & $A_1$ \\
$\xi_{13} = v_{1111} + v_{0121}+v_{1220} + v_{1122}$ & $16$ &  $1$ & $\emptyset$ \\
$\xi_{14} = v_{1110} + v_{0111}+v_{0122}$ & $16$ &  $1$ & $A_1$ \\
$\xi_{15} = v_{1110} + v_{1120}+v_{0122}$ & $16$ &  $1$ & $A_1$ \\
$\xi_{16} = v_{1110} + v_{0121}+v_{1220}+v_{0122}$ & $14$ &  $1$ & $\emptyset$ \\
$\xi_{17} = v_{1110} + v_{0111}+v_{1120}+v_{0122}$ & $12$ &  $S_3$ & $\emptyset$ \\
$\xi_{18} = v_{0010} + v_{0001}+v_{1220}$ & $12$ &  $1$ & $A_1$ \\
$\xi_{19} = v_{0121} + v_{1000} + v_{0100}$ & $12$ &  $1$ & $A_1$ \\
$\xi_{20} = v_{0010} + v_{0001} + v_{1220} +v_{1122}$& $10$ &  $1$ & $\emptyset$ \\
$\xi_{21} =  v_{1111} +v_{0121} + v_{1000} + v_{0100}$& $10$ &  $1$ & $\emptyset$ \\
$\xi_{22} = v_{0110} +v_{0011} + v_{1100} + v_{0120} $ & $8$ &  $1$ & $\emptyset$ \\
$\xi_{23} = v_{0110} +v_{0001} + v_{1000} + v_{0120} $& $6$ &  $1$ & $\emptyset$ \\
$\xi_{24} = v_{0010} +v_{0001}+ v_{1000} + v_{0100} $& $4$ &  $1$ & $\emptyset$ \\
\end{tabular}
\end{table}
Recall that for any variety $X$ defined over $\mathbb{F}_q$, we denote by $F = Fr_q: X \rightarrow X$ the geometric Frobenius. If $X$ is a $G$-variety and $x \in X^F$ such that the stabilizer $G_x$ is connected, then the number of $\mathbb{F}_q$-points in the corresponding $G$-orbit $G \cdot x$ can easily be computed using the following well-known lemma (see for example \cite[Lemma 4.2.13]{digne2020representations}).
\begin{lemma}\label{lemma: number of points in G/H with H connected}
    Let $H \subset G$ be a connected subgroup (defined over $\mathbb{F}_q$). Then $\# (G/H)^F =\frac{\# G^F}{\# H^F}$.
\end{lemma}
For disconnected stabilizers, counting the $\mathbb{F}_q$-points of an orbit is more complicated. However, if the connected stabilizer is unipotent, we still have the following result.
\begin{lemma}\label{lemma: number of points in G/H with H0 unipotent}
    Let $H \subset G$ be a subgroup (defined over $\mathbb{F}_q$) such that $H^{\circ}$ is unipotent. Then $\# (G/H)^{F} = \frac{\# G^F}{q^{\dim H}}$.
\end{lemma}
\begin{proof}
    There is a bijection
    \begin{equation*}
       \{ G^F\text{-orbits on } (G/H)^F \} \overset{1:1}{\leftrightarrow} \{F\text{-conjugacy classes in } H/H^{\circ} \}
    \end{equation*}
    which identifies the $F$-conjugacy class of an element $\bar{\sigma} \in H/H^{\circ}$ (where $\sigma \in H$) with the orbit of $gH \in (G/H)^F$ where $g \in G$ satisfies $\sigma = g^{-1}F(g)$ (c.f. \cite[Prop. 4.2.14]{digne2020representations}). Fix $\sigma \in H$ with $g$ as above. We first compute the $G^F$-stabilizer of $gH \in (G/H)^F$ and thus the size of the $G^F$-orbit of $gH$. Then we sum over all the orbits to compute the size of $(G/H)^F$.
    
    We write $F_{\sigma}$ for the map $x \mapsto \sigma F(x) \sigma^{-1}$. Let $\tau_1, ... \tau_k \in H $ be coset representatives of $H/H^{\circ}$. Then the $G$-stabilizer of $gH \in G/H$ is
    \begin{equation*}
        G_{gH} = gHg^{-1} = \bigsqcup_{i=1}^k g\tau_i H^{\circ} g^{-1}.
    \end{equation*}
    Note that $F(g) = g \sigma$ and thus
    \begin{equation*}
        F(g\tau_i H^{\circ}g^{-1}) = g \sigma F(\tau_i) H^{\circ} (g \sigma)^{-1} =  g F_{\sigma}(\tau_i) H^{\circ} g^{-1}.
    \end{equation*}
    Hence, if $(g\tau_i H^{\circ} g^{-1})^F \neq \emptyset$, we have $\bar{\tau}_i \in (H/H^{\circ})^{F_{\sigma}}$. In that case, we can write $F_{\sigma}(\tau_i) = \tau_i x$ for some $x \in H^{\circ}$. Then for any $h \in H^{\circ}$ we have
    \begin{equation*}
        F(g\tau_i h g^{-1}) = g \sigma F(\tau_ih) \sigma^{-1} g^{-1} = g F_{\sigma}(\tau_i h) g^{-1} =  g \tau_i x  F_{\sigma}(h)  g^{-1}.
    \end{equation*}
    Thus, $g\tau_i h g^{-1}$ is fixed under $F$ if and only if $h = x F_{\sigma}(h)$. By Lang's theorem we can find $y \in H^{\circ}$ with $y^{-1} F_{\sigma}(y) = x$. Then we can identify
    \begin{align*}
        (g \tau_i H^{\circ} g^{-1})^F \overset{1:1}&{\leftrightarrow} \{ h \in H^{\circ} \mid h = x   F_{\sigma}(h) \} \\
        \overset{1:1}&{\leftrightarrow} \{ h \in H^{\circ} \mid yh = F_{\sigma} (yh) \} \\
        \overset{1:1}&{\leftrightarrow} (H^{\circ})^{F_{\sigma}}.
    \end{align*}
    Since $H^{\circ}$ is unipotent, the set $(H^{\circ})^{F_{\sigma}}$ has $q^{\dim H^{\circ}} =q^{\dim H}$ many elements (c.f. \cite[Prop. 4.1.12]{digne2020representations}). This shows that the number of points in the $G^F$-stabilizer of $gH$ is
    \begin{equation*}
        \# (G^F)_{gH} = (G_{gH})^F = \sum_{i=1}^k \# (g \tau_i H^{\circ} g^{-1})^F = \# (H/H^{\circ})^{F_{\sigma}} \cdot q^{\dim H}.
    \end{equation*}
   Note that $(H/H^{\circ})^{F_{\sigma}}$ is precisely the stabilizer of $\bar{\sigma} \in H/H^{\circ}$ under the $F$-conjugation action. Hence, $\sum_{\bar{\sigma} \in  \{F\text{-conj. cl. in } H/H^{\circ} \}}  \frac{1}{ \# (H/H^{\circ})^{F_{\sigma}}} = 1$ by the orbit stabilizer theorem. Combining these results, we get
    \begin{align*}
        \# (G/H)^F &= \sum_{\mathcal{O} \in \{  G^F\text{-orbits in } (G/H)^F \}}  \# \mathcal{O}\\
        &= \sum_{\substack{ \bar{\sigma} \in  \{F\text{-conj. cl. in } H/H^{\circ} \} \\ \sigma = g^{-1}F(g)}} \# (G^F \cdot gH ) \\
        & = \sum_{\bar{\sigma} \in  \{F\text{-conj. cl. in } H/H^{\circ} \}}  \frac{\# G^F}{ \# (H/H^{\circ})^{F_{\sigma}}  \cdot q^{\dim H}} \\
        &= \frac{\# G^F}{  q^{\dim H}}.
    \end{align*}
\end{proof}
For any $i = 1,...,24$, let $\mathcal{O}_i := G \cdot \xi_i$ be the orbit of the element $\xi_i$ from \cref{table: nilpotent orbits}.
\begin{lemma}\label{lemma: Fq points of orbits are as in table}
    For each $i = 1,...,24$, the number of points in $\mathcal{O}_i^F$ are as in \cref{table: Fq points of orbits orbits}.
\end{lemma}
\begin{proof}
    This follows directly from \cref{lemma: number of points in G/H with H connected,lemma: number of points in G/H with H0 unipotent} using \cref{table: nilpotent orbits} and the standard formula for the $\mathbb{F}_q$-points of a reductive group (\cite[Prop. 4.4.1]{digne2020representations}).
\end{proof}
\begin{theorem}\label{thm: finitely many orbits}
    There are only finitely many $G$-orbits on $\mathfrak{N}(V)$. The elements listed in \cref{table: nilpotent orbits} are a complete list of orbit representatives in $\mathfrak{N}(V)$.
\end{theorem}
\begin{proof}
    We need to show that any element of $\mathfrak{N}(V)$ lies in one of the orbits $\mathcal{O}_i$ and that these orbits are distinct. A straightforward computation shows that
    \begin{equation*}
        \sum_{i = 1}^{24} \# \mathcal{O}_i^F \overset{\cref{lemma: Fq points of orbits are as in table}}{=} q^{48} \overset{\cref{prop: number of F2n points of nilcone}} = \# \mathfrak{N}(V)^F.
    \end{equation*}
    Since this holds for any $q=2^n$, we get that any element lies in one of the orbits $\mathcal{O}_i$ once we can show that the $\mathcal{O}_i$ are pairwise distinct.
    
    In most cases, we can immediately see that the $\mathcal{O}_i$ are pairwise distinct by comparing the corresponding stabilizer structures using \cref{lemma: stabilizers are as in table}. It remains to check that
    \begin{equation*}
        \xi_5 \not\sim \xi_6 , \quad \xi_8 \not\sim \xi_9, \quad \xi_{11} \not\sim \xi_{12}, \quad \xi_{14} \not\sim \xi_{15}, \quad \xi_{18} \not\sim \xi_{19}, \quad \xi_{20} \not\sim \xi_{21}.
    \end{equation*}
    For $\xi_5 \not\sim \xi_6$ this is clear since $\xi_5 \in \mathfrak{g}_s$ and $\xi_6 \in \mathfrak{g}/\mathfrak{g}_s$. For $\xi_8 \not\sim \xi_9$ we use that their respective long components (i.e. the component of $\mathfrak{g}/\mathfrak{g}_s$ in $V = \mathfrak{g}_s \oplus \mathfrak{g}/\mathfrak{g}_s$) are $v_{2342} = \xi_2$ and $v_{1220} + v_{1122} = \xi_6$ which are not conjugate as we have already seen above (by comparing stabilizer structures). Similarly, $\xi_{11} \not\sim \xi_{12}$ since the long component of $\xi_{11}$ is $v_{1220}$ which is conjugate to $\xi_2$ but the long component of $\xi_{12}$ is $v_{1220} +v_{1122} = \xi_6$. By a similar argument $\xi_{14} \not\sim \xi_{15}$ since the long component of $\xi_{14}$ is $v_{0122}$ which is conjugate to $\xi_2$ but the long component of $\xi_{15}$ is $v_{1120} + v_{0122}$ which is conjugate to $\xi_6$. For $\xi_{18} \not\sim \xi_{19}$ we use that the long component of $\xi_{18}$ is $v_{1220}$ which is conjugate to $\xi_2$ but the long component of $\xi_{19}$ is $v_{1000} + v_{0100}$ which is conjugate to $\xi_6$. Thus, it remains to show that $\xi_{20} \not\sim \xi_{21}$. Assume that $g \xi_{20} = \xi_{21}$. Write $g = b\dot{w}b'$ with $b,b' \in B$ and $w \in W$. Then we have $\dot{w} b' \xi_{20} = b^{-1}\xi_{21}$. Looking at the weight spaces that contribute to $\xi_{20}$ and $\xi_{21}$, we get that
    \begin{equation*}
        w(\alpha_{0001}), w(\alpha_{0010}), w(\alpha_{1220}), w(\alpha_{1122})\in  \mathbb{Z}_{\ge 0} \Phi^- +  \{ \alpha_{1111}, \alpha_{0121}, \alpha_{1000}, \alpha_{0100}\}.
    \end{equation*}
    A direct computation (either by hand or using a computer) show that now such $w$ can exist. This shows that $\xi_{20} \not\sim \xi_{21}$ which completes the proof.
\end{proof}

\begin{table}[ht]
\caption{$\mathbb{F}_q$-points of orbits in $\mathfrak{N}(V)$}
\label{table: Fq points of orbits orbits}
\centering
\begin{tabular}{c|c}
    Representative $\xi_i$ & \# $\mathcal{O}_{i}^F$ \\\hline 
    $\xi_1$ & $1$ \\
    $\xi_2$ & $(q^4+1)(q^{12}-1)$ \\
    $\xi_3$ & $(q^4+1)(q^{12}-1)$ \\
    $\xi_4$ & $(q^4+1)(q^6-1)(q^{12}-1)$ \\
    $\xi_5$ & $q^4(q^8-1)(q^{12}-1)$ \\
    $\xi_6$ & $q^4(q^8-1)(q^{12}-1)$ \\
    $\xi_7$ & $q^4(1+q^2+q^4)(q^8-1)(q^{12}-1)$ \\
    $\xi_8$ & $q^4(q^6-1)(q^8-1)(q^{12}-1)$ \\
    $\xi_9$ & $q^4(q^6-1)(q^8-1)(q^{12}-1)$ \\
    $\xi_{10}$ & $q^{10}(q^4+1)(q^6-1)(q^{12}-1)$ \\
    $\xi_{11}$ & $q^{8}(q^6-1)(q^8-1)(q^{12}-1)$ \\
    $\xi_{12}$ & $q^{8}(q^6-1)(q^8-1)(q^{12}-1)$ \\
    $\xi_{13}$ & $q^{8}(q^2-1)(q^6-1)(q^8-1)(q^{12}-1)$ \\
    $\xi_{14}$ & $q^{10}(q^6-1)(q^8-1)(q^{12}-1)$ \\
    $\xi_{15}$ & $q^{10}(q^6-1)(q^8-1)(q^{12}-1)$ \\
    $\xi_{16}$ & $q^{10}(q^2-1)(q^6-1)(q^8-1)(q^{12}-1)$ \\
    $\xi_{17}$ & $ q^{12}(q^2-1)(q^6-1)(q^8-1)(q^{12}-1)$ \\
    $\xi_{18}$ & $q^{14}(q^6-1)(q^8-1)(q^{12}-1)$ \\
    $\xi_{19}$ & $q^{14}(q^6-1)(q^8-1)(q^{12}-1)$ \\
    $\xi_{20}$ & $q^{14}(q^2-1)(q^6-1)(q^8-1)(q^{12}-1)$ \\
    $\xi_{21}$ & $q^{14}(q^2-1)(q^6-1)(q^8-1)(q^{12}-1)$ \\
    $\xi_{22}$ & $q^{16}(q^2-1)(q^6-1)(q^8-1)(q^{12}-1)$ \\
    $\xi_{23}$ & $q^{18}(q^2-1)(q^6-1)(q^8-1)(q^{12}-1)$ \\
    $\xi_{24}$ & $q^{20}(q^2-1)(q^6-1)(q^8-1)(q^{12}-1)$ \\
\end{tabular}
\end{table}
\section{Affine pavings of Springer fibers}
Recall that an affine paving of a variety $X$ is a sequence
\begin{equation*}
    \emptyset = X_0 \subset X_1 \subset X_2 \subset .... \subset X_n = X
\end{equation*}
such that each $X_i$ is closed in $X$ and $X_i \backslash X_{i-1}$ is isomorphic to an affine space $\mathbb{A}^{n_i}$ for some $n_i \ge 0$. The flag variety $\mathcal{B}$ has an affine paving with closed subspaces $\mathcal{B}_{\le w} = \bigsqcup_{y \le w} B \dot{y}B$ where $\le$ is a total order on $W$ extending the Bruhat order. Our goal in this section is to prove that the `exotic Springer fibers'
\begin{equation*}
    \mathcal{B}_{x} := \mu^{-1}(x) \cong \{  gB \in \mathcal{B} \mid g^{-1}x \in V^{\le 0} \}
\end{equation*}
admit an affine paving for any $x \in \mathfrak{N}(V)$.

Any $v \in V^-$ can be written as $v = \lambda_1 v_{\beta_1} + \lambda_2 v_{\beta_2} + ... + \lambda_k v_{\beta_k}$ for some distinct $\beta_1, ... , \beta_k \in \Phi^-$ and $\lambda_1, ... , \lambda_k \in \bar{\mathbb{F}}_2^{\times}$. We can then define
\begin{align*}
    \Phi_{\ge v} &:= \bigcup_{i =1 }^k \{ \beta \in   (\mathbb{Z}_{\ge 0} \Phi^- + \beta_i ) \cap \Phi^- \mid \beta \text{ and } \beta_i \text{ have the same length}\} \\
    V_{\ge v} &:= \bigoplus_{\beta \in \Phi_{\ge v}} V_{\beta}.
\end{align*}
Note that by \cref{lemma: action of U on V} we have
\begin{equation*}
    B \cdot v \subset V_{\ge v}.
\end{equation*}
The set $\Phi_{\ge v}$ can easily be read off from \cref{figure: roots in F4} as the set of all roots above the $\beta_i$ in the corresponding short or long graph. Moreover, one can compute the dimension of the stabilizer $B_v$ using \cref{lemma: action of U on V} (either by hand or with the help of a computer). For the elements $\xi_1,...,\xi_{24}$ we have listed $ \dim B_{\xi_i}$ and $\# \Phi_{\ge \xi_i}$ in \cref{table: stabilizers in B}. We have also listed in \cref{table: stabilizers in B} for each $\xi_i$ a distinguished cocharacter where a tuple $(a,b,c,d)$ corresponds to the cocharacter $a \omega_1 + b\omega_2 + c \omega_3 + d \omega_4$ with $\langle \omega_i, \alpha_j \rangle = \delta_{ij}$. These will be used in the proof of \cref{thm: affine paving}.

\begin{table}[ht]
\caption{$B$-stabilizers, roots above $\xi_i$ and cocharacters}
\label{table: stabilizers in B}
\centering
\begin{tabular}{c|c|c|c}
    Representative $\xi_i$ &  $\dim B_{\xi_i}$ & $ \# \Phi_{\ge \xi_i}$ & cocharacter $\lambda_i$ \\\hline 
    $\xi_1$ & $28$ & $0$ & $(1,1,1,1)$\\
    $\xi_2$ & $27$ & $1$ & $(1,1,1,1)$\\
    $\xi_3$ & $27$ & $1$ & $(1,1,1,1)$\\
    $\xi_4$ & $26$ & $2$ & $(1,1,1,1)$\\
    $\xi_5$ & $22$ & $6$ & $(1,1,1,1)$\\
    $\xi_6$ & $22$ & $6$ & $(1,2,1,1)$\\
    $\xi_7$ & $22$ & $6$ & $(1,1,1,1)$\\
    $\xi_8$ & $21$ & $7$ & $(1,1,1,1)$\\
    $\xi_9$ & $20$ & $8$ & $(1,2,1,1)$\\
    $\xi_{10}$ & $16$ & $12$ & $(1,1,1,1)$\\
    $\xi_{11}$ & $17$ & $11$ & $(1,1,1,1)$\\
    $\xi_{12}$ & $17$ & $11$ & $(1,2,1,1)$\\
    $\xi_{13}$ & $16$ & $12$ & $(1,2,1,1)$\\
    $\xi_{14}$ & $14$ & $14$ & $(1,1,1,1)$\\
    $\xi_{15}$ & $14$ & $14$ & $(2,1,1,1)$\\
    $\xi_{16}$ & $14$ & $14$ & $(3,1,1,2)$\\
    $\xi_{17}$ & $12$ & $16$ & $(0,1,1,0)$ \\
    $\xi_{18}$ & $11$ & $17$ & $(1,1,1,1)$\\
    $\xi_{19}$ & $11$ & $17$ & $(1,1,1,1)$\\
    $\xi_{20}$ & $10$ & $18$ & $(1,2,1,1)$\\
    $\xi_{21}$ & $10$ & $18$ & $(1,1,1,1)$\\
    $\xi_{22}$ & $8$ & $20$ & $(2,1,1,1)$\\
    $\xi_{23}$ & $6$ & $22$ & $(3,1,1,2)$\\
    $\xi_{24}$ & $4$ & $24$ & $(1,1,1,1)$\\
\end{tabular}
\end{table}
\begin{corollary}\label{corollary: B orbit is dense}
    For each $\xi_i$ ($i =1 ,..., 24$), we have $\overline{B \cdot \xi_i} = V_{\ge \xi_i}$.
\end{corollary}
\begin{proof}
    It follows from \cref{table: stabilizers in B} that for each $i = 1,...,24$, we have $\dim B_{\xi_i} + \# \Phi_{\ge \xi_i} = 28$. Since $\dim B = 28$, this implies $\dim B \cdot \xi_i = \dim V_{\ge \xi_i}$. By the irreducibility of $V_{\ge \xi_i}$, this implies $\overline{B \cdot \xi_i} = V_{\ge \xi_i}$
\end{proof}
The following result is a basic application of \cite[Lemma 2.2]{de1988homology}.
\begin{corollary}\label{corollary: springer fiber intesect with Bruhat cell is affine}
    For each $i=1,...,24$ and $w \in W$ the intersection $\mathcal{B}_{\xi_i} \cap B\dot{w}B/B$ is smooth.
\end{corollary}
\begin{proof}
    There is a canonical isomorphism $B/(B \cap {}^wB) \cong B\dot{w}B/B$. Using this, we can identity
    \begin{align*}
        \mathcal{B}_{\xi_i} \cap B\dot{w}B/B & \cong \{ b \in B \mid  b^{-1}\xi_i \in w(V^{\ge 0}) \}/ (B \cap {}^w B) \\
        &=  \{ b \in B \mid  b^{-1}\xi_i \in w(V^{\ge 0}) \cap V_{\ge \xi_i} \}/ (B \cap {}^w B).
    \end{align*}
    Since $B \cdot \xi_i \subset V_{\ge \xi_i}$ is dense (\cref{corollary: B orbit is dense}), the smoothness of $\mathcal{B}_{\xi_i} \cap B\dot{w}B/B$ follows from \cite[Lemma 2.2]{de1988homology} (applied in the notation of \textit{loc. cit.} with $V = V_{\ge \xi _i}$, $U = w(V^{\ge 0}) \cap V_{\ge \xi_i}$, $M = B$ and $H = B \cap {}^w B$).
\end{proof}
To construct affine pavings, we will make use of the following result.
\begin{lemma}\label{lemma: affine fibtraion over affine space is affine space}
    Let $X \rightarrow \mathbb{A}^n$ be a morphism of affine varieties such that all fibers are isomorphic to an affine space $\mathbb{A}^k$. Then $X \cong \mathbb{A}^{n+k}$.
\end{lemma}
\begin{proof}
    Any morphism of affine varieties whose fibers are affine spaces comes from a vector bundle by \cite{bass1976locally,suslin1977locally}. Moreover, any vector bundle over an affine space is trivial by the Quillen-Suslin theorem \cite{quillen1976projective}. Thus, $X \rightarrow \mathbb{A}^n$ is a trivial vector bundle, so $X \cong \mathbb{A}^{n+k}$. 
\end{proof}
We now show that all exotic Springer fibers admit an affine paving.
\begin{theorem}\label{thm: affine paving}
    For each $i=1,...,24$ the variety $\mathcal{B}_{\xi_i}$ admits an affine paving. More specifically:
    \begin{enumerate}
        \item For $i \neq 17$ and any $w \in W$, the intersection $\mathcal{B}_{\xi_i} \cap B\dot{w}B/B$ is either empty an affine space.
        \item For any $w \in W$, the intersection $\mathcal{B}_{\xi_{17}} \cap B\dot{w}B/B$ is either empty or a disjoint union of at most two affine spaces. If there are two connected components, then $w^{-1}(\alpha_1), w^{-1}(\alpha_2) \not\in \Phi^-$ and one component is the attracting locus of $\dot{w}B/B $ in $\mathcal{B}_{\xi_{17}}$ and the other is the attracting locus of $ u \dot{w}B/B $ in $\mathcal{B}_{\xi_{17}}$ under the $\mathbb{G}_m$-action induced by the cocharacter $\lambda_{17} = (0,1,1,0)$ (where $u$ is as in \eqref{eq: def of u and s}).
    \end{enumerate}
\end{theorem}
\begin{proof}
    Let $\xi_i = \xi_i^s + \xi_i^l$ where $\xi_i^s \in \mathfrak{g}_s$ and $\xi_i^l \in \mathfrak{g} / \mathfrak{g}_s$. We have listed in \cref{table: stabilizers in B} for each $\xi_i$ a cocharacter $\lambda_i : \mathbb{G}_m \rightarrow T$ where a tuple $(a,b,c,d)$ corresponds to the cocharacter $a\omega_1 + b \omega_2 + c \omega_3 + d \omega_4$ with $\langle \omega_i , \alpha_j \rangle = \delta_{ij}$. Looking at the weight spaces appearing in each $\xi_i$ one can check that $\lambda_i(t) \xi_i = t^{n_i} \xi_i^s + t^{m_i} \xi_i^l$ for some $m_i, n_i >0$. Hence,
    \begin{equation*}
        \mathcal{B}_{\xi_i} = \mathcal{B}_{\xi_i^s} \cap \mathcal{B}_{\xi_i^l} = \mathcal{B}_{t^{n_i}\xi_i^s} \cap \mathcal{B}_{t^{m_i} \xi_i^l} = \mathcal{B}_{\lambda_i(t) \xi_i}
    \end{equation*}
    for all $t \in \mathbb{G}_m$. Thus, each $\mathcal{B}_{\xi_i}$ is a $\mathbb{G}_m$-stable subvariety of $\mathcal{B}$ where $\mathbb{G}_m$ acts via $\lambda_i$.
    
    To prove (1), let us now consider the case where $i \neq 17$. Then the cocharacter $\lambda_i$ in \cref{table: stabilizers in B} is antidominant, i.e. for each $\alpha \in \Phi^-$ we have $\alpha(\lambda_i(t)) = t^k$ for some $k > 0$. It follows from this that
    \begin{equation*}
        \mathcal{B}^{\mathbb{G}_m} = \{\dot{w} B \mid w \in W \}
    \end{equation*}
    and $\lambda_i$ contracts $B\dot{w}B/B$ onto $\{ \dot{w}B/B \}$. Note that $\mathcal{B}_{\xi_i}$ is closed under taking $xB \mapsto \lim_{t \rightarrow 0} \lambda_i(t)xB$ since $\mathcal{B}_{\xi_i} \subset \mathcal{B}$ is closed. Hence, $\mathcal{B}_{\xi_{17}} \cap B \dot{w} B$ is either empty or $\dot{w}B/B  \in \mathcal{B}_{\xi_{17}}$ and $\mathcal{B}_{\xi_{17}} \cap B \dot{w} B$ is the attracting locus of $\{ \dot{w}B/B \}$ in $\mathcal{B}_{\xi_{17}}$. Note that $\mathcal{B}_{\xi_i} \cap B\dot{w}B/B$ is smooth by \cref{corollary: springer fiber intesect with Bruhat cell is affine}. Thus, if $\mathcal{B}_{\xi_i} \cap B\dot{w}B/B$ is non-emptry, it follows from results of Białynicki-Birula \cite{bialynicki1973some} that $\mathcal{B}_{\xi_i} \cap B\dot{w}B/B$ is an affine fibration over $ \{ \dot{w}B/B \}$. In other words, $\mathcal{B}_{\xi_i} \cap B\dot{w}B/B$ is either empty or an affine space. This proves (1).

    To prove (2), we now consider the case where $i = 17$. We use the cocharacter $\lambda_{17} = (0,1,1,0)$. Note that $\lambda_{17}$ is not antidominant since we only have $\alpha(\lambda_{17}(t)) = t^k$ for some $k >0$ if $\alpha \in \Phi^- \backslash \{\alpha_{1}, \alpha_4 \}$. On the other hand, $\alpha_{1}(\lambda_{17}(t)) = \alpha_4(\lambda_{17}(t)) = 1$. Hence, we have
    \begin{equation*}
        (B \dot{w} B/B)^{\mathbb{G}_m} = \begin{cases}
            U_{\alpha_{1}} U_{\alpha_4} \dot{w}B/B & w^{-1}(\alpha_{1}),  w^{-1}(\alpha_4) \not\in \Phi^-\\ 
            U_{\alpha_{1}}  \dot{w}B/B & w^{-1}(\alpha_{1})\not\in \Phi^- ,  w^{-1}(\alpha_4) \in \Phi^-\\ 
            U_{\alpha_4} \dot{w}B/B &  w^{-1}(\alpha_{1})\in \Phi^- ,  w^{-1}(\alpha_4) \not\in \Phi^-\\ 
            \{ \dot{w}B/B \} & w^{-1}(\alpha_{1}),  w^{-1}(\alpha_4) \in \Phi^-
        \end{cases}
    \end{equation*}
    and $\lambda_{17}$ contracts $B\dot{w}B/B$ onto $(B\dot{w}B/B)^{\mathbb{G}_m}$. Using \cref{remark: computation via graph} one can easily compute that for any $a, b \in \bar{\mathbb{F}}_2$, we have
    \begin{equation*}
        x_{\alpha_{1}}(a) x_{\alpha_4}(b) \xi_{17} = \xi_{17} + (a+b) v_{1111} + (a+b^2) v_{1122}.
    \end{equation*}
    Thus, $(\mathcal{B}_{\xi_{17}} \cap B \dot{w} B/B)^{\mathbb{G}_m}$ is either empty or isomorphic to a closed subvariety of $\mathbb{A}^2 \cong U_{\alpha_{1}} U_{\alpha_4}$ with defining ideal generated by a subset of $\{a,b,a+b,a+b^2 \}$ (where $a \in \mathcal{O}(U_{\alpha_{1}} U_{\alpha_4})$ picks out the $U_{\alpha_{1}} $ component and $b$ the $U_{\alpha_{4}}$ component). The only possibility for this is that $(\mathcal{B}_{\xi_{17}} \cap B \dot{w} B/B)^{\mathbb{G}_m}$ is either empty, an affine space or consists of the two points $\{ \dot{w}B/B, u \dot{w}B/B \}$. It follows again from \cref{corollary: springer fiber intesect with Bruhat cell is affine} and results of Białynicki-Birula \cite{bialynicki1973some} that $\mathcal{B}_{\xi_{17}} \cap B \dot{w} B/B \rightarrow (\mathcal{B}_{\xi_{17}} \cap B \dot{w} B/B)^{\mathbb{G}_m}$ is an affine fibration. Thus, if $(\mathcal{B}_{\xi_{17}} \cap B \dot{w} B/B)^{\mathbb{G}_m}$ is an affine space we can deduce that $\mathcal{B}_{\xi_{17}} \cap B \dot{w} B/B$ is an affine space by \cref{lemma: affine fibtraion over affine space is affine space}. In the case where $(\mathcal{B}_{\xi_{17}} \cap B \dot{w} B/B)^{\mathbb{G}_m}$ consists of two points, we get that $\mathcal{B}_{\xi_{17}} \cap B \dot{w} B/B$ is a disjoint union of two affine spaces which are the attracting loci of $\dot{w} B/B$ and $u\dot{w}B/B$ in $\mathcal{B}_{\xi_{17}}$.
\end{proof}
Note that for any variety $X$ with an affine paving $\emptyset = X_0 \subset X_1 \subset .... \subset X_n = X$, the Borel-Moore homology $H_*(X)$ has a basis given by the fundamental classes $[X_1], [X_2], ..., [X_n]$ and thus $H_{odd}(X) = 0$. Thus, we obtain the following corollary from \cref{thm: affine paving}.
\begin{corollary}\label{corollary: odd BM vanishing}
    For each $i = 1,...,24$, we have $H_{odd}(\mathcal{B}_{\xi_i}) = 0$.
\end{corollary}
\section{The Springer correspondence}\label{section: Springer corr}
We now establish our exotic Springer correspondence. Consider the `exotic Springer sheaf'
\begin{equation*}
   \textbf{S} := \mu_* \textbf{1}_{\tilde{V}}[\dim \tilde{V}] \in D^b_c(\mathfrak{N}(V)).
\end{equation*}
This is a $G$-equivariant perverse sheaf since $\mu$ is $G$-equivariant and semismall (\cref{lemma: springer res is small/semismall}). The following result was also obtained in the proof of \cite[Lemma 4.10]{antor2025geometric} from a geometric realization of the affine Hecke algebra with unequal parameters. We give a more direct proof which is analogous to the Fourier transform approach to classical Springer theory (see for example \cite[Thm. 8.2.8]{achar2021perverse}).
\begin{proposition}\label{prop: endomorphisms of springer sheaf}
    There is a canonical isomorphism $\mathbb{C}[W] \cong \End_{D^b_c(\mathfrak{N}(V))}(\textbf{S})$.
\end{proposition}
\begin{proof}
    Recall that $\mathfrak{g}_s$ is the irreducible $G$-module whose non-zero weights are the short roots. The same holds for the dual representation $\mathfrak{g}_s^{\vee}$ and thus we can pick an isomorphism of $G$-modules $\mathfrak{g}_s \cong \mathfrak{g}_s^{\vee}$. Similarly, $\mathfrak{g}/\mathfrak{g}_s$ is the irreducible $G$-module whose weights are the long roots and we get an isomorphism $\mathfrak{g}/\mathfrak{g}_s \cong (\mathfrak{g}/\mathfrak{g}_s)^{\vee}$. Combining these yields an isomorphism of $G$-modules $V \overset{\sim}{\rightarrow} V^{\vee}$. This isomorphism identifies $V_{\alpha}$ with $(V^{\vee})_{\alpha} = (V_{- \alpha})^{\vee}$ and thus it restricts to an isomorphism
    \begin{equation}\label{eq: negative part identified with annihilator}
        V^- \overset{\sim}{\rightarrow} (V^{\le 0})^{\perp} = \{ \lambda \in V^{\vee} \mid \lambda|_{V^{\le 0}} = 0 \}.
    \end{equation}
    The rest of the proof is analogous to the Fourier transform approach to classical Springer theory (c.f. \cite[Thm. 8.2.8]{achar2021perverse}). We sketch the argument for the convenience of the reader. If $\mathcal{V} \rightarrow X$ is a vector bundle with dual bundle $\mathcal{V}^{\vee} \rightarrow X$, the Fourier-Laumon transform (see \cite[§6.9]{achar2021perverse}) is an equivalence of categories
    \begin{equation*}
        Four_V : D^b_{\mathbb{G}_m}(\mathcal{V}) \rightarrow D^b_{\mathbb{G}_m}(\mathcal{V}^{\vee}).
    \end{equation*}
    We regard $\tilde{V}$ and $\tilde{V}^{gs}$ as a subbundles of the trivial bundle $G \times^B V \cong \mathcal{B} \times V$. Denote the corresponding inclusions by $\iota$ and $\iota^{gs}$. Using the isomorphism $V \cong V^{\vee}$, we can identify $G \times^B V$ with its dual bundle $(G \times^B V)^{\vee} = G \times^B V^{\vee}$. It follows from \eqref{eq: negative part identified with annihilator} that under this identification, the annihilator bundle $(\tilde{V}^{gs})^{\perp} = G \times^B (V^{\le 0})^{\perp}$ gets identified with $\tilde{V} = G \times^B V^-$. It then follows from \cite[Cor. 6.9.14]{achar2021perverse} that
    \begin{equation}\label{eq: fourier of gs bundle is springer res}
        Four_{G \times^B V}(\iota^{gs}_* \textbf{1}_{\tilde{V}^{gs}}[\dim \tilde{V}^{gs}]) = \iota_* \textbf{1}_{\tilde{V}}[\dim \tilde{V}].
    \end{equation}
    Note that there is a cartesian square
    \begin{equation*}
        \begin{tikzcd}
            G \times^B V \cong \mathcal{B} \times V \arrow[r, "p_2"] \arrow[d, "p_1"] & V \arrow[d] \\
            \mathcal{B} \arrow[r] & \{ pt \}
        \end{tikzcd}
    \end{equation*}
    and $p_2 \circ \iota^{gs} = \mu^{gs}$ and $p_2 \circ \iota = i \circ \mu$ where $i: \mathfrak{N}(V) \hookrightarrow V$ is the inclusion. Moreover, by \cite[Prop. 6.9.15]{achar2021perverse} we have $(p_2)_* \circ Four_{G \times^B V} \cong Four_{V} \circ (p_2)_*$. Thus
    \begin{align*}
        i_* \textbf{S} & \cong(p_2)_* \iota_* \textbf{1}_{\tilde{V}} [\dim \tilde{V}] \\
        \overset{\eqref{eq: fourier of gs bundle is springer res}}&{\cong}  (p_2)_*Four_{G \times^B V}(\iota^{gs}_* \textbf{1}_{\tilde{V}^{gs}}[\dim \tilde{V}^{gs}]) \\
        &\cong Four_{V}((p_2)_*\iota^{gs}_* \textbf{1}_{\tilde{V}^{gs}}[\dim \tilde{V}^{gs}]) \\
        &\cong Four_{V}(\mu^{gs}_* \textbf{1}_{\tilde{V}^{gs}}[\dim \tilde{V}^{gs}]).
    \end{align*}
    Hence, we can identify
    \begin{align*}
        \End(\textbf{S}) & \cong \End(i_* \textbf{S}) \\
        & \cong \End(Four_{V}(\mu^{gs}_* \textbf{1}_{\tilde{V}^{gs}}[\dim \tilde{V}^{gs}])) \\
        & \cong \End(\mu^{gs}_* \textbf{1}_{\tilde{V}^{gs}}[\dim \tilde{V}^{gs}]) \\
        \overset{\cref{corollary: grothendieck springer sheaf is IC sheaf}}&{\cong}  \mathbb{C}[W].
    \end{align*}
\end{proof}
The isomorphism $\mathbb{C}[W] \cong \End(\textbf{S})$ from \cref{prop: endomorphisms of springer sheaf} gives rise to a canonical bijection
\begin{equation}\label{eq: Irr W bijects to general perverse sheaves}
    \Irr(W) \overset{1:1}{\leftrightarrow} \{ X \in \Irr(\Perv_G(\mathfrak{N}(V))) \mid X \text{ is a direct summand of } \textbf{S} \}.
\end{equation}
Using that there are only finitely many $G$-orbits in $\mathfrak{N}(V)$ (\cref{thm: finitely many orbits}) there is a bijection
\begin{align*}
    \Irr(\Perv_G(\mathfrak{N}(V))) \overset{1:1}&{\leftrightarrow}\{ (x,\rho) \mid  x \in \mathfrak{N}(V) , \rho \in \Irr(A(x)) \} /G \\
    IC(G\cdot x, \rho) &\mapsto (x, \rho).
\end{align*}
If $M$ is an $A(x)$-representation and $\rho$ is an irreducible $A(x)$-representation, we define
\begin{equation*}
    M_{\rho} := \Hom_{A(x)}(\rho, M).
\end{equation*}
\begin{lemma}
    If $IC(G\cdot x, \rho)$ appears in $\textbf{S}$, then $H_*(\mathcal{B}_x)_{\rho} \neq 0$.
\end{lemma}
\begin{proof}
    For any variety $X$, let $p:X \rightarrow \{ pt \}$ be the structure map. Since $\tilde{V}$ is smooth, we have $\omega_{\tilde{V}} = \textbf{1}_{\tilde{V}} [2 \dim \tilde{V}]$. Let $\iota: \mathcal{B}_x \hookrightarrow \mathcal{B}$ and $\iota_x : \{ x \} \hookrightarrow \mathfrak{N}(V)$ be the natural inclusions. By base change, we have
    \begin{align*}
        H_i(\mathcal{B}_x) & = H^{-i}(p_* \omega_{\mathcal{B}_x}) \\
        & = H^{-i}(p_* \iota^! \textbf{1}_{\tilde{V}} [2 \dim \tilde{V}]) \\
        & = H^{-i} ( \iota_x^! \mu_* \textbf{1}_{\tilde{V}}[2 \dim \tilde{V}]) \\
        &= H^{ 2 \dim \tilde{V} -i }(\iota_x^! \textbf{S}).
    \end{align*}
    If $IC(G\cdot x, \rho)$ appears in $\textbf{S}$, then $\rho$ appears in $H^*(\iota_x^!\textbf{S})$ and thus also in $H_*(\mathcal{B}_x)$. This completes the proof.
\end{proof}
Using \eqref{eq: Irr W bijects to general perverse sheaves} and the lemma above, we get a canonical injection
\begin{equation}\label{eq: Springer injection}
    \Irr(W) \hookrightarrow \{ (x,\rho) \mid x \in \mathfrak{N}(V) , \rho \in \Irr(A(x)), H_*(\mathcal{B}_x)_{\rho} \neq 0 \}/G.
\end{equation}
We now show that this map is a bijection.
\begin{theorem}\label{thm: exotic Springer correspondence}
    The map from \eqref{eq: Springer injection} is a bijection. We have $H_*(\mathcal{B}_x)_{\rho} = 0$ if and only if $(x, \rho) \not\sim (\xi_{17}, sgn)$.
\end{theorem}
\begin{proof}
    It is well-known that $\Irr(W)$ has $25$ elements, so the right hand side of \eqref{eq: Springer injection} has at least $25$ elements. On the other hand, we can see from \cref{table: nilpotent orbits} that the right hand side of \eqref{eq: Springer injection} has at most $26$ elements since there are $26$ pairs of the form $(x,\rho)$ with $x \in \mathfrak{N}(V)$ and $\rho \in \Irr(A(x))$ up to conjugacy. Thus, to prove that \eqref{eq: Springer injection} is a bijection, it suffices to show that there is an orbit and a local system with $H_*(\mathcal{B}_{x})_{\rho} = 0$ (which then has to be unique up to conjugacy).
    
    We claim that $H_*(\mathcal{B}_{\xi_{17}})_{sgn} = 0$. Recall from \cref{lemma: stabilizers are as in table} that $A(\xi_{17}) = \langle \bar{u}, \bar{s} \rangle \cong S_3$. Also, recall that we have constructed in \cref{thm: affine paving} an affine paving of $\mathcal{B}_{\xi_{17}}$. More specifically let $x_1, ..., x_n \in W$ be the elements such that $ \mathcal{B}_{\xi_{17}} \cap B\dot{x}_iB/B =: X_i$ is an affine space and let $y_1 ,..., y_m \in W$ be the elements such that $\mathcal{B}_{\xi_{17}} \cap B\dot{y}_j B/B$ is a disjoint union of two affine spaces $Y_{j,1} $ and $Y_{j,2} $ where $\dot{w} B/B \in Y_{j,1}$ and $u \dot{w} B/B \in Y_{j,2}$. Then by \cref{thm: affine paving} we get a distinguished basis of $H_*(\mathcal{B}_{\xi_{17}})$ consisting of fundamental classes:
    \begin{equation*}
        \{ [\bar{X}_i], [\bar{Y}_{j,1}], [\bar{Y}_{j,2}] \mid i=1,...,n \text{ } j=1,...,m \}.
    \end{equation*}
    Note that any $\mathcal{B}_{\xi_{17}} \cap B\dot{w}B/B$ is $B_{\xi_{17}}$-stable. Since $u \in B_{\xi_{17}}$, we see that the elements $[\bar{X}_i]$ are fixed by $u$. Thus, we get a morphism of $S_3$-representations
    \begin{align*}
        \Ind^{S_3}_{\{e, \bar{u} \}} triv & \rightarrow H_*(\mathcal{B}_{\xi_{17}}) \\
        1 \otimes 1 & \mapsto [\bar{X}_i].
    \end{align*}

    Recall from \cref{thm: affine paving} that $Y_{j,1}$ is the attracting locus of $\dot{y}_j B /B \in \mathcal{B}_{\xi_{17}}$ and $Y_{j,2}$ is the attracting locus of $u \dot{y}_j B/B \in \mathcal{B}_{\xi_{17}}$ under the $\mathbb{G}_m$-action on $\mathcal{B}_{\xi_{17}}$ induced by $\lambda_{17} = (0,1,1,0)$. The element $sus = x_{-\alpha_1}(1) x_{-\alpha_4}(1)$ clearly commutes with $\lambda_{17}(t)$. Hence, $sus$ maps $Y_{j,1}$ to the attracting locus of $sus\dot{y}_jB/B$. We have
    \begin{equation*}
        sus\dot{y}_jB/B =  x_{-\alpha_1}(1) x_{-\alpha_4(1)} \dot{y}_j B/B = \dot{y}_jB/B
    \end{equation*}
    since $y_j^{-1}(-\alpha_1), y_j^{-1}(-\alpha_4) \in \Phi^-$ by \cref{thm: affine paving}. Thus, $sus$ maps $Y_{j,1}$ to $Y_{j,1}$ and hence it fixes the element $[\bar{Y}_{j,1}]$. In particular, we get a morphism of $S_3$-representations
    \begin{align*}
        \Ind^{S_3}_{\{e, \overline{sus} \}} triv & \rightarrow H_*(\mathcal{B}_{\xi_{17}}) \\
        1 \otimes 1 & \mapsto [\bar{Y}_{j,1}].
    \end{align*}
    Similarly, the element $s = u(sus)u$ commutes with $\lambda_{17} = (0,1,1,0)$ and fixes the element $u\dot{y}_jB/B \in \mathcal{B}_{17}$. Thus, it stabilizes the corresponding attracting locus which shows that $s$ fixes $[\bar{Y}_{j,2}]$. Hence, we get a morphism of $S_3$-representations
    \begin{align*}
        \Ind^{S_3}_{\{e, \bar{s} \}} triv & \rightarrow H_*(\mathcal{B}_{\xi_{17}}) \\
        1 \otimes 1 & \mapsto [\bar{Y}_{j,2}].
    \end{align*}
    Combining these maps, we get a surjective map of $S_3$-representations
    \begin{equation*}
        (\Ind^{S_3}_{\{e, \bar{u} \}} triv)^{\oplus n} \oplus (\Ind^{S_3}_{\{e, \overline{sus} \}} triv)^{\oplus m} \oplus (\Ind^{S_3}_{\{e, \bar{s} \}} triv)^{\oplus m} \twoheadrightarrow H_*(\mathcal{B}_{\xi_{17}}).
    \end{equation*}
    The sign representation never appears in the induction of the trivial representation from an order $2$ subgroup of $S_3$. Hence, $H_*(\mathcal{B}_{\xi_{17}})_{sgn} = 0$ which completes the proof.
\end{proof}
We conclude by discussing some applications to the representation theory of affine Hecke algebras. Let $\mathcal{H}_{q_1, q_2}^{\aff}(F_4)$ be the affine Hecke algebra of type $F_4$ with two formal parameters (where the short roots have parameter $q_1$ and the long roots have parameter $q_2$). The torus
\begin{equation*}
    \hat{T} := T \times \mathbb{G}_m \times \mathbb{G}_m
\end{equation*}
naturally acts on $V, \mathfrak{N}(V)$ and $\mathcal{B}$ where $T$ acts in the natural way, the first $\mathbb{G}_m$ acts by scaling $\mathfrak{g}_s$ and the second by scaling $\mathfrak{g}/\mathfrak{g}_s$ (and both copies of $\mathbb{G}_m$ act trivially on $\mathcal{B}$). We can also consider the corresponding complex and positive real torus
\begin{align*}
    \hat{\mathcal{T}} &:= (\mathbb{C}^{\times} \otimes_{\mathbb{Z}} X_*(T)) \times \mathbb{C}^{\times} \times \mathbb{C}^{\times} \\
    \hat{\mathcal{T}}_{\mathbb{R}_{>0}} &:= (\mathbb{R}_{>0} \otimes_{\mathbb{Z}} X_*(T)) \times \mathbb{R}_{>0} \times \mathbb{R}_{>0}.
\end{align*}
Then there is a canonical bijection (c.f. \cite[§3.2]{antor2025geometric})
\begin{align*}
    \hat{\mathcal{T}}/W \overset{1:1}&{\leftrightarrow} \{\text{Central characters } Z(\mathcal{H}_{q_1,q_2}^{\aff}(F_4)) \rightarrow \mathbb{C} \} \\
    a & \mapsto \chi_a.
\end{align*}
The central character $\chi_a$ is called positive real if $a \in \hat{\mathcal{T}}_{\mathbb{R}_{>0}}$. Note that we can canonically identify
\begin{equation*}
    X^*(\hat{\mathcal{T}}) \cong X^*(\hat{T}).
\end{equation*}
Using this, we can define for any $a \in \hat{\mathcal{T}}$ a diagonalizable subgroup scheme of $\hat{T}$
\begin{equation*}
    \hat{T}_a := Spec(\bar{\mathbb{F}}_2[X^*(\hat{T}) / \{ \rho \in X^*(\hat{T})\mid \rho(a) =1  \}]).
\end{equation*}
When $a \in \hat{\mathcal{T}}_{\mathbb{R}_{>0}}$ is positive real, $\hat{T}_a$ is a subtorus of $\hat{T}$ (c.f. \cite[Lemma 4.7]{antor2025geometric}).
\begin{theorem}\label{thm: exotic DL correspondence}
    Let $a \in \hat{\mathcal{T}}_{\mathbb{R}_{>0}}$ be a positive real element. Then there is a bijection
    \begin{equation*}
        \Irr_{\chi_a}(\mathcal{H}^{\aff}_{q_1,q_2}(F_4)) \overset{1:1}{\leftrightarrow} \{ (x,\rho) \mid x \in \mathfrak{N}(V)^{\hat{T}_a}, \rho \in \Irr( A(a,x)), H_*(\mathcal{B}_x^{\hat{T}_a})_{\rho} \neq 0 \}/G^{\hat{T}_a}
    \end{equation*}
    where $A(a,x) := G^{\hat{T}_a}_x/ (G^{\hat{T}_a}_x)^{\circ}$ and $\Irr_{\chi_a}(\mathcal{H}^{\aff}_{q_1,q_2}(F_4))$ denotes the set of all irreducible $\mathcal{H}^{\aff}_{q_1,q_2}(F_4)$-representations with central character $\chi_a$.
\end{theorem}
\begin{proof}
    This follows from \cite[Thm. 4.13]{antor2025geometric}. The conditions (A1), (A2) and (A3) from \textit{loc. cit.} are satisfied by \cref{thm: finitely many orbits,corollary: odd BM vanishing,thm: exotic Springer correspondence}.
\end{proof}
\begin{remark}
    We expect that the theorem above holds without the assumption that $a = (s,c_1,c_2) \in \hat{\mathcal{T}}$ is positive real as long as $\langle c_1,c_2 \rangle \subset \mathbb{C}^{\times}$ is torsion-free. By \cite{antor2025geometric} this comes down to verifying similar conditions (A1), (A2), (A3) for some exotic nilcones associated to various reductive subgroups of $F_4$. In fact, we expect these conditions to hold in much greater generality (see the introduction of \cite{antor2025geometric}) and for the relevant reductive subgroups of $F_4$ one should be able to verify these conditions by direct computation using techniques similar to the ones of this paper.
\end{remark}

\bibliographystyle{alpha}
\bibliography{bibliography}
\end{document}